\documentclass[11pt,reqno]{amsart}
\usepackage{graphicx} 
\usepackage[dvipsnames]{xcolor}
\usepackage{todonotes}
\usepackage[foot]{amsaddr}
\usepackage{smartdiagram}
\usepackage{tikz}
\usepackage{amsmath,bm,bbm,amsthm, amssymb}
\usepackage{fullpage}
\usepackage{color}
\usepackage{dsfont}
\usepackage{animate}
\usepackage{etoolbox}
\usepackage{comment}
\usepackage{pgf}
\usetikzlibrary{calc}
\usetikzlibrary{patterns}
\usetikzlibrary{arrows}
\usetikzlibrary{decorations.pathreplacing}
\usepackage[utf8]{inputenc}
\usepackage{pgfplots}

\usepackage[most]{tcolorbox}
\usepackage{adjustbox}
\usepackage[final]{pdfpages}
\usepackage[ruled,vlined,linesnumbered,algosection,resetcount]{algorithm2e}
\usepackage{blkarray}
\usepackage{wrapfig}
\usepackage{ marvosym }
\usepackage[hidelinks]{hyperref}
\usepackage{orcidlink}
\usepackage{nicefrac}

\newtheorem{theorem}{Theorem}[section]
\newtheorem{corollary}[theorem]{Corollary}
\newtheorem{lemma}[theorem]{Lemma}

\newtheorem{proposition}[theorem]{Proposition}
\newtheorem{remark}[theorem]{Remark}

\numberwithin{equation}{section}

\newcommand{\supp}{\mathrm{supp}}

\author{Markus Heydenreich\orcidlink{0000-0002-3749-7431}}
\address[Markus Heydenreich]{Institut für Mathematik, Universität Augsburg, 86135 Augsburg, Germany}
\email{markus.heydenreich@uni-a.de}
\author{Christian Hirsch\orcidlink{0000-0003-4136-3740}}
\address[Christian Hirsch]{Department of Mathematics, Aarhus University, Ny Munkegade 118, 8000 Aarhus C, Denmark}
\address[Christian Hirsch]{DIGIT Center, Aarhus University, Finlandsgade 22, 8200 Aarhus N, Denmark}
\email{hirsch@math.au.dk}
\author{Matthias L\"owe\orcidlink{0000-0001-9893-3770}}
\address[Matthias Löwe]{Institut für Mathematische Stochastik, Universität Münster, Orléans-Ring 10, 48149 Münster}
\title{A mathematical analysis of hierarchical Hopfield models}
\email{maloewe@uni-muenster.de}

\begin{document}
	\maketitle
    \begin{abstract}
        {The central question that we address is: How can structured 
        information be stored in a hierarchical Hopfield model involving hidden layers? To this end, we develop a formalism of \emph{strokes} and \emph{concepts} that allows us to appropriately structure information: 
        initial features are first classified into \emph{strokes}, which in a second step are aggregated into \emph{concepts}. 
        We rigorously derive criteria under which concepts can be retrieved from noisy input data. A remarkable effect is that we do not require a perfect retrieval at the level of strokes, as the second-layer retrieval procedure compensates for first-layer errors. 
        We treat separately the cases of fixed and variable-sized concepts. }
    \end{abstract}
	\section{Introduction}

	Associative memory refers to the cognitive capacity to recall interconnected sets of information, called memories or patterns, when presented with partial cues from one such set. This mechanism underlies numerous human cognitive functions.
	
	This idea was mathematically formalised by Hopfield \cite{Hopfield1982} in his groundbreaking 1982 paper. {In their original model, a set of $N$ neurons are in either of the two states $-1$ and $+1$.}  {We will instead use the equivalent formulation where states are encoded by $0$ and $+1$, which is more common in the framework of sparse patterns and the Willshaw model, see e.g.\ \cite{Amari1989,griponb,GHLV16,Willshaw}.} 
    These neurons are connected and each of the connections between neuron $i$ and neuron $j$ is endowed with a random variable $J_{ij}$ (assuming the patterns are chosen at random).  
	A state $\sigma$ of this model is a collection $\sigma=(\sigma_i)_{i=1}^N\in \{-1,+1\}$ respectively\ $\{0,1\}$ depending on the above choice of the activities. Such a state is endowed with an energy function; in the Hopfield model this energy is given by 
	$$
	H_N(\sigma):=-\frac 1 {2N} \sum_{i,j=1}^N \sigma_i \sigma_j J_{ij}.
	$$
    For patterns $(\xi^\mu)_{\mu=1}^N=((\xi_i^\mu)_{i=1}^N)_{\mu=1}^M$  
    the $J_{ij}$ are defined as $J_{ij}=\sum_{\mu=1}^M \xi_i^\mu \xi_j^\mu$. 
    Typical pattern choices are such that 
the $((\xi_i^\mu)_{i=1}^N)_{\mu=1}^M\in\{-1,+1\}^{N\times M} $ are centered and i.i.d., which reflects
    the absence of any a~priori structural information about the patterns.
    If $M$ is not too large, the patterns minimise the energy function and are therefore fixed points of a gradient descent dynamics. Moreover, a noisy input can be reconstructed under such a dynamics. 
	
	A central question in this context is: \emph{How large can $M$ be chosen to ensure that patterns are (with high probability) fixed points of such a retrieval dynamics?} For the classical Hopfield model the answer is 
	$M =\Theta (N/\log N)$ if we require perfect reconstruction \cite{L98,MPRV}, and $M=\Theta(N)$ {in the situation where it suffices to almost reconstruct a given pattern} \cite{AGS,loukianova,Newman_hopfield,Talagrand_hopfield}. For sparse patterns, this capacity can be increased to $M = \Theta( N^2/\log N)$ \cite{GHLV16}. Here we use the Landau notation, i.e., {$f \in \Theta(g)$ whenever $c < \inf_x f(x)/g(x)\le\sup_x f(x)/g(x) <C$ for uniform constants $C>c>0$.}
	
	Certain applications demand memory capacity far exceeding $\mathrm{const.} N$, and the classical Hopfield network proves inadequate for such scenarios due to its limited capacity. If, however, we change the quadratic interaction terms $\sigma_i\sigma_j$ in the energy function to a polynomial or even exponential term, then in the resulting setup one can indeed store a polynomial or even super-polynomial number of patterns, cf.\ \cite{albanese2026yet,DHLUV17,KrotovHopfield2016}, {and also  \cite{gayrard2025} for new results on stable reconstruction of these models}. 
	Krotov and Hopfield \cite{Krotov2020LargeAM} argue that such \emph{Dense Associative Memories} are biologically not plausible, because they cannot be designed in terms of true microscopic degrees of freedom, since they contain many-body interaction terms. 

	Instead, they propose a hierarchical model and show that  Dense Associative Memories are effective descriptions of a theory involving additional hidden neurons with only two-body synaptic interactions. These additional hidden layers allow for the construction of biologically plausible networks with the same capacity as Dense Associative Memories.
	
	Such hierarchical Hopfield models are the focus of this paper. 
	The key question we address is: \emph{What advantages does such a hierarchical construction provide beyond reproducing known capacity results?}
    It appears that hierarchical constructions do not, by themselves, improve the asymptotic storage capacity beyond classical Hopfield scaling. Instead, their core effect lies in enabling different mechanisms for organizing and retrieving patterns. This will be made explicit in a concrete hierarchical stroke–concept model studied below.

{    Hierarchical and hierarchically structured associative memories have also been studied earlier in the
literature, both in terms of explicit hierarchical architectures and in terms of hierarchically
correlated or ultrametrically organised pattern ensembles, see
\cite{BacciMatoParga1990,BosKuehnVanHemmen1988,CortesKroghHertz1987,Gutfreund1988,KroghHertz1988,PargaVirasoro1986}.
Our perspective here is different: we focus on a concrete compositional stroke--concept model and
derive explicit probabilistic retrieval guarantees for this hierarchical construction.}

    We now summarise our findings.	After introducing the hierarchical stroke--concept model, we first analyse the case in which each concept consists of a fixed number of strokes. In this setting, we prove sufficient conditions for successful retrieval of a prescribed concept and derive asymptotic consequences for the admissible growth of the number of stored concepts. We then consider an extension in which the concept sizes are allowed to vary. For this variable-size model, we formulate an appropriate penalised scoring rule and establish analogous retrieval results on suitably sized windows. These results give a mathematical description of how retrieval performance depends on both the feature dimension and the distribution of concept sizes.

    \medskip
	\paragraph{\bf Organisation.} In Section 2 we introduce the hierarchical model and its energy formulation. Section 3 analyses the resulting effective energies and establishes capacity benchmarks. In Section 4 we introduce a hierarchical stroke–concept construction. Section 5 contains the main results {for equal concept sizes,} and the proofs are collected in Section 6. {Finally, Section 7 discusses the case of variable-sized concepts.}

\section{Model and Energy Formalism}

We consider a class of hierarchical associative memory models consisting of a visible
(feature) layer and one or more hidden layers, and focus on the effective energy induced on the visible units. Throughout, we focus on the energy
structure induced on the visible units, {and on the consequences of such a hierarchical setup for storage
capacity.}

{Similarly to the setup considered in \cite{KrotovHopfield2016},} the network consists of a visible layer with $N_f$ units and a collection of $N_h$ hidden units arranged in one or more layers. For clarity of exposition, we first describe a single hidden layer and later comment on deeper hierarchies. The visible units are described by variables $x_i \in \mathbb{R}$, $i=1,\dots,N_f$, while the hidden units are denoted by $y_\mu \in \mathbb{R}$, $\mu=1,\dots,N_h$.  The real-valued formulation allows for a convenient energy-based description; discrete models arise as special cases.
The connectivity between visible and hidden units is specified by a weight matrix
\[
\Xi = (\xi_{\mu i}) \in \mathbb{R}^{N_h \times N_f}.
\]
No direct couplings within the visible layer or within the hidden layer are assumed at
the microscopic level.
The network dynamics is of leaky--integrator type and compatible with a global Lyapunov function. 
Putting 
\[
u_\mu(x) = \sum_{i=1}^{N_f} \xi_{\mu i} x_i
\]
for the overlaps between the visible state and the hidden patterns, 
a convenient formulation is obtained by introducing an energy
functional $E(x,y)$ of the form
\begin{equation}
	E(x,y)
	= \frac12 \sum_{i=1}^{N_f} x_i^2
	+ \sum_{\mu=1}^{N_h} U(y_\mu)
	- \sum_{\mu=1}^{N_h} u_\mu(x)\,\phi(y_\mu),
	\label{eq:general-energy}
\end{equation}
which consists of a quadratic regularisation on the visible layer, a local potential on the hidden units, and an interaction term coupling visible and hidden variables via overlaps.
 The function $U$
denotes a hidden-layer potential and $\phi = U'$ its associated activation function.
The quadratic term $\tfrac12 \sum_i x_i^2$ plays the role of a leak or regularisation on
the visible layer.

Under gradient-descent dynamics,
\[
\tau_x \frac{dx_i}{dt} = -\frac{\partial E}{\partial x_i},
\qquad
\tau_y \frac{dy_\mu}{dt} = -\frac{\partial E}{\partial y_\mu},
\]
and the energy $E(x,y)$ decreases monotonically along trajectories. Fixed points of the
dynamics therefore correspond to local minima of $E$.

A central assumption is separation of time scales such that the hidden variables relax much faster
than the visible ones. In this regime, for each fixed visible configuration $x$, the
hidden variables $y$ remain close to the minimiser of $E(x,y)$ with respect to $y$. Since
the energy \eqref{eq:general-energy} is a sum of independent single-site contributions in
the hidden variables, this minimisation decouples over $\mu$.

This motivates defining an effective energy on the visible layer by
\begin{equation}
	E_{\mathrm{eff}}(x)
	:= \inf_{y \in \mathbb{R}^{N_h}} E(x,y),
	\label{eq:effective-energy-def}
\end{equation}
which governs the slow dynamics of the visible units.
The structure of the resulting effective Hamiltonian depends
crucially on the choice of the hidden-layer potential $U$.

If $U$ is quadratic, elimination of the hidden variables produces an effective energy
that is purely quadratic in the visible variables. In this case the model reduces to a
Hopfield-type associative memory with an effective coupling matrix determined by the
weights $\Xi$ and, in deeper hierarchies, by additional hidden-layer couplings \cite{Krotov2020LargeAM}. As a
consequence, the storage capacity is bounded by the classical Hopfield scaling, growing
at most linearly with the number $N_f$ of visible units.

By contrast, genuinely nonlinear choices of $U$ lead, after elimination of the hidden
layer, to effective energies that depend nonlinearly on the overlaps $u_\mu(x)$. These
nonlinear contributions are the essential mechanism underlying high-capacity
associative memories. The precise form of the resulting effective energy and its
implications for storage capacity are discussed in the following sections.

\section{Effective Energies and Capacity Benchmarks}

In this section, we analyse the class of effective visible-layer energies that arise from
eliminating hidden layers in hierarchical associative memory models. Our goal is to identify the structural features of the effective Hamiltonian that control storage capacity. In particular, we clarify the role of quadratic versus higher-order hidden-layer potentials and establish the
Hopfield model as the natural capacity benchmark for a broad class of hierarchical constructions. This perspective mainly serves to clarify the role of nonlinearity in hierarchical models and to provide structural background for the constructions studied in the following sections.

We first consider the case in which all hidden-layer potentials are quadratic. In this
setting, the minimisation of the energy with respect to the hidden variables can be
carried out explicitly \cite{Krotov2020LargeAM} and yields an effective energy on the visible layer that is purely
quadratic in the visible activities. Independently of the number of hidden layers or the
details of their couplings, the resulting effective Hamiltonian takes the generic form
\begin{equation}
	E_{\mathrm{eff}}(x)
	= \frac12 \sum_{i=1}^{N_f} x_i^2
	- \frac12 \sum_{i,j=1}^{N_f} J^{\mathrm{eff}}_{ij}\, x_i x_j ,
	\label{eq:quadratic-effective-energy}
\end{equation}
for some symmetric coupling matrix $J^{\mathrm{eff}}$ determined by the hierarchical
weights.

When the visible units are restricted to binary states $x_i \in \{-1,1\}$, the quadratic
self-energy term is constant and the model reduces to a Hopfield-type Hamiltonian. It
follows that, as long as the effective energy remains quadratic, the storage capacity is
bounded by the classical Hopfield scaling and grows at most linearly with the number
$N_f$ of visible units. Deeper hierarchies may reshape correlations among the effective
patterns encoded in $J^{\mathrm{eff}}$, but they do not alter the asymptotic capacity
scaling.

To go beyond quadratic effective energies, nonlinearity must survive the elimination of
the hidden variables. The simplest setting in which this occurs is a model with a single
hidden layer endowed with a nonlinear potential. Here, the effective energy on
the visible layer depends on the overlaps
\[
u_\mu(x) = \sum_{i=1}^{N_f} \xi_{\mu i} x_i
\]
through a nonlinear function.
More precisely, for a hidden-layer potential $U$ with activation $\phi = U'$, elimination
of the hidden variables leads to an effective visible-layer energy of the form
\begin{equation}
	E_{\mathrm{eff}}(x)
	= \frac12 \sum_{i=1}^{N_f} x_i^2
	- \sum_{\mu=1}^{N_h} \Phi\bigl(u_\mu(x)\bigr),
	\label{eq:nonlinear-effective-energy}
\end{equation}
where the function $\Phi$ is defined by
$
\Phi(u)
= - \inf_{y \in \mathbb{R}} \bigl[ U(y) - u\,\phi(y) \bigr].
$
This class of energies coincides with the nonlinear Hopfield-type models introduced in
the context of Dense Associative Memories \cite{DHLUV17, KrotovHopfield2016}. The growth behaviour of $\Phi(u)$ for large
$|u|$ determines the stability of stored patterns and the achievable storage capacity.

Introducing additional hidden layers and subsequently freezing them modifies the effective function $\Phi$ appearing in \eqref{eq:nonlinear-effective-energy}, cf.\ \cite{krotov2021hierarchical}.  However, it does
not change the basic structure of the energy, which remains a sum of independent contributions depending only on the overlaps $u_\mu(x)$. 
From a mathematical point of view, any effective nonlinearity $\Phi$ generated by a finite hierarchy under suitable
regularity assumptions can also be realised by a single hidden layer equipped with an
appropriate potential $U$.
Consequently, increasing the depth of the hierarchy does not, by itself, enlarge the
class of effective Hamiltonians on the visible layer, nor does it change the asymptotic
storage capacity scaling as a function of $N_f$. Capacity enhancement beyond the
Hopfield bound requires genuinely nonlinear dependence on the overlaps, not merely
additional quadratic layers.

This indicates a clear separation of roles. Quadratic hidden layers, even
in deep hierarchies, lead to Hopfield-type effective energies and are therefore subject to
the classical capacity bound. Nonlinear hidden-layer potentials are necessary to achieve
superlinear storage capacity. At the same time, depth may still play an important role in
shaping the geometry of basins of attraction, the robustness of retrieval, or the
structure of correlations among patterns.

In the following section, we investigate these questions in detail and show that, while
the asymptotic storage capacity remains unchanged, hierarchical architectures can yield
substantial qualitative improvements in associative memory performance.
We now turn to a different aspect, namely how hierarchical structure can affect retrieval performance in concrete models.

%
%
\section{A Hierarchical Stroke--Concept Memory}
\label{sec:hier}

This section introduces the hierarchical construction which is the backbone of our main results.
We consider a three-layer network consisting of a feature layer and two associative layers, where the second associative layer also serves as the output layer.

Neurons of the 
input layer are connected to the neurons of the first associative layer and neurons of the first associative layer are connected to neurons on the second associative layer. Other connections do not exist.

The key idea of this construction is to represent patterns in a compositional way: simple “stroke” features are combined to form more complex “concepts”. This hierarchical organisation allows for reliable retrieval even in the presence of local errors, by aggregating many weak signals across multiple strokes.

The first associative layer stores \emph{primitive} patterns (so-called ``strokes'') in feature
space, while a second associative layer stores \emph{compositions} of these primitives
(so-called ``concepts''). Our focus is on the regime in which concepts are sparse combinations of
strokes and retrieval proceeds from a concept-level feature cue.

We illustrate the model in Figure \ref{fig:overlap_calc}. While the precise notation used in that figure will be introduced later in this section, the reader may appreciate a visual illustration of the model before the notation is introduced in full.

\subsection{Strokes and concepts}

Let $N_f$  be the number of feature units. We store $M$ stroke patterns
\[
\xi^\mu = (\xi_i^\mu)_{i=1}^{N_f} \in \{0,1\}^{N_f}, \qquad \mu=1,\dots,M,
\]
which we assume to be i.i.d.\ sparse with
\begin{equation}
	\label{eq:stroke-sparsity}
	{p:=}\mathbb P(\xi_i^\mu=1)=\frac{\log N_f}{N_f}.
\end{equation}
Thus, a typical stroke has Hamming weight $|\xi^\mu|_1 \approx \log N_f$. We remark that it is basically this sparsity assumption that motivates the choice of $0$ and $+1$ as 
the binary values a neuron in a stroke may take. On the one hand this resembles the setup 
in Willshaw or Amari models \cite{Willshaw}, and on the other hand it is a  convenient choice for determining which strokes get activated by a feature input.

 \usetikzlibrary{matrix,arrows.meta,decorations.pathreplacing}

The feature-level representation of a concept $\alpha$ is the pixelwise OR (=union) of its
constituent strokes,
\begin{equation}
	\label{eq:concept-or}
    \eta^{(\alpha)}=\big(\eta^{(\alpha)}_i\big)_{i=1,\dots,N_f}
    \quad\text{with}\quad
	\eta^{(\alpha)}_i := \bigvee_{\nu\in S_\alpha} \xi_i^\nu. 
\end{equation}
where $S_\alpha$ is a subset of strokes that are active in the concept $\alpha$.

\subsection{Layer-wise retrieval rule}
Let us next describe how retrieval in our network works. 
Given a feature-level cue $\tilde \eta\in\{0,1\}^{N_f}$, the first (stroke) layer produces
a binary activation vector $x\in\{0,1\}^M$ via thresholding overlaps:
\begin{equation}
	\label{eq:stroke-threshold-rule}
	x_\mu
	= \mathbf 1\!\left\{
	O_\mu(\tilde\eta)\ge \theta
	\right\},
	\qquad \text{where 
	$O_\mu(\tilde\eta):=\sum_{i=1}^{N_f}\tilde\eta_i\,\xi_i^\mu$
    and 
	$\theta=\kappa\log N_f$, $0<\kappa<1$.}
\end{equation}
{This step can be interpreted as a noisy detection of the underlying stroke features: some true strokes may be missed, and some spurious ones may be activated. The subsequent decoding stage will aggregate this information across strokes.} Intuitively, this step identifies which stored
strokes are present in the given cue.
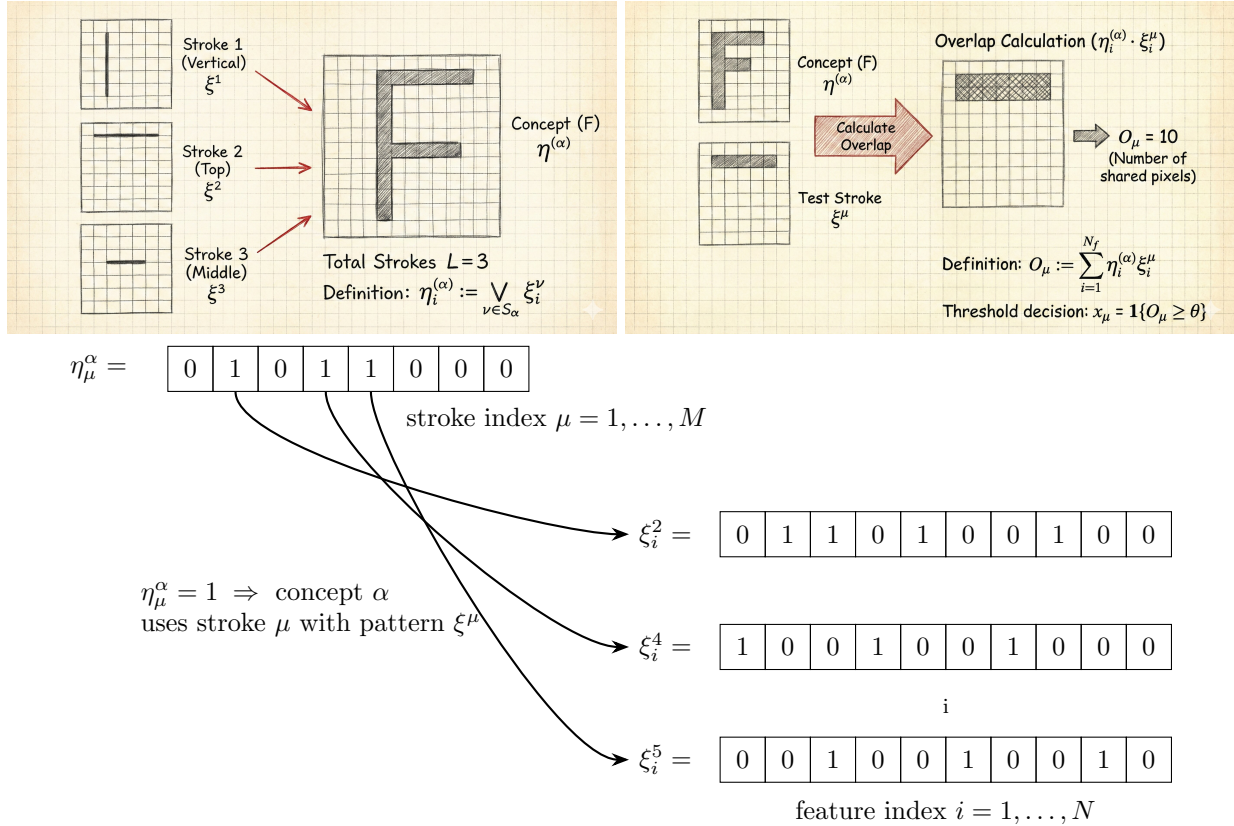
\begin{figure}[btp!]
    \centering
    \includegraphics[width=0.49\linewidth]{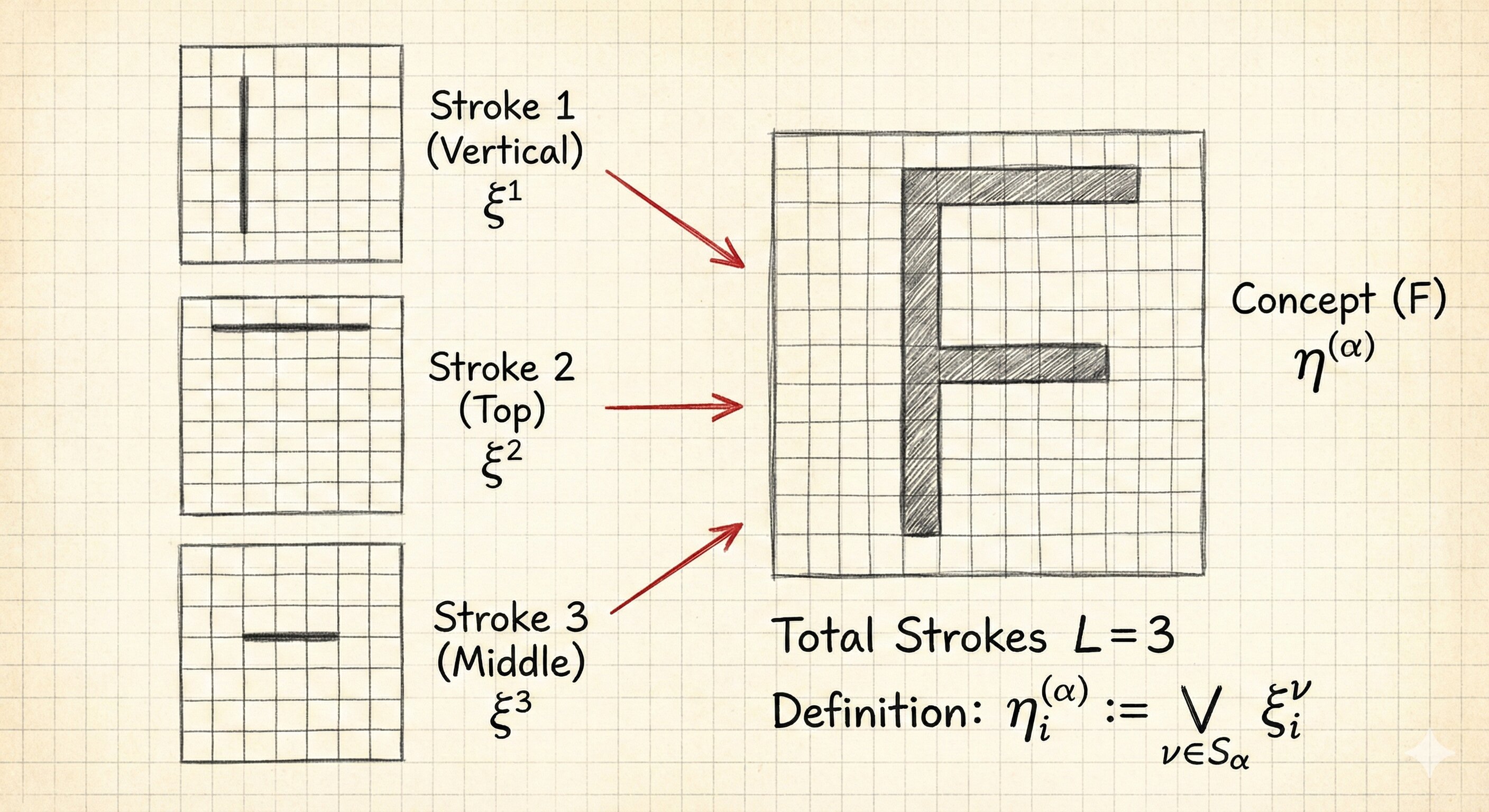}
    \includegraphics[width=0.49\linewidth]{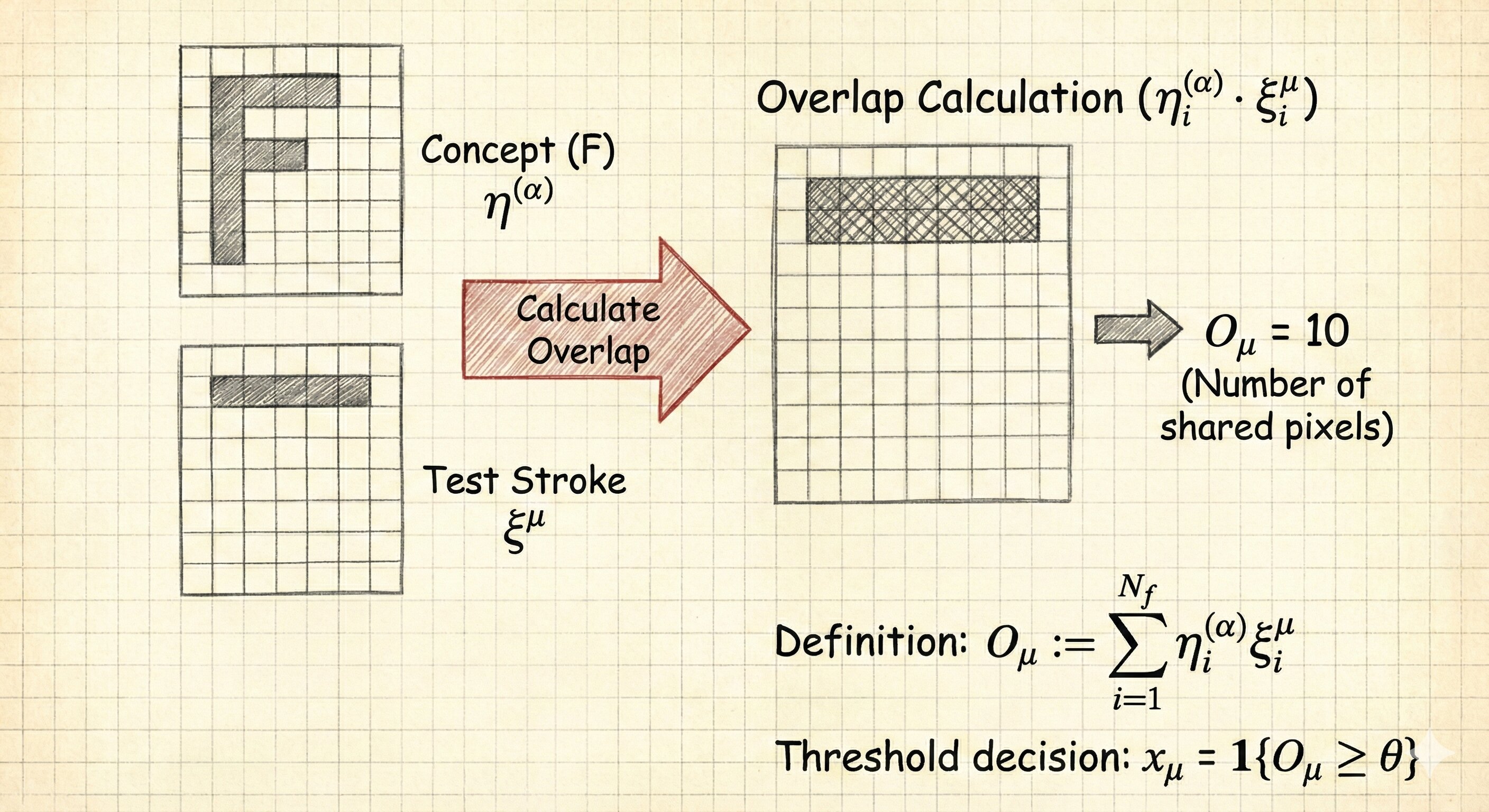}\\
\begin{tikzpicture}[
    every node/.style={font=\small},
    box/.style={draw, minimum width=6mm, minimum height=6mm, anchor=center},
    >=Stealth
]


\matrix (eta) [matrix of nodes,
    nodes={box},
    column sep=-\pgflinewidth
] {
    0 & 1 & 0 & 1 & 1 & 0 & 0 & 0 \\
};

\node[left=4mm of eta-1-1] {$\eta^\alpha_\mu =$};

\matrix (muind) [matrix of nodes,
    nodes={anchor=north,font=\scriptsize,inner sep=0pt},
    below=1mm of eta
] {
     &  &  &  &  &  &  &  \\
};

\node[right=5mm of muind] {stroke index $\mu = 1,\dots,M$};


\node[left] (lab2) at (4.7,-2.2) {$\xi^{2}_i =$};
\matrix (xi2) [matrix of nodes,
    nodes={box},
    column sep=-\pgflinewidth,
    right=1mm of lab2
] {
    0 & 1 & 1 & 0 & 1 & 0 & 0 & 1 & 0 & 0 \\
};

\node[left] (lab4) at (4.7,-3.7) {$\xi^{4}_i =$};
\matrix (xi4) [matrix of nodes,
    nodes={box},
    column sep=-\pgflinewidth,
    right=1mm of lab4
] {
    1 & 0 & 0 & 1 & 0 & 0 & 1 & 0 & 0 & 0 \\
};

\node[left] (lab5) at (4.7,-5.2) {$\xi^{5}_i =$};
\matrix (xi5) [matrix of nodes,
    nodes={box},
    column sep=-\pgflinewidth,
    right=1mm of lab5
] {
    0 & 0 & 1 & 0 & 0 & 1 & 0 & 0 & 1 & 0 \\
};

\matrix (iind) [matrix of nodes,
    nodes={anchor=north,font=\scriptsize,inner sep=0pt},
    below=1mm of xi4
] {
     &   &   &   & i &   &   &   &   &  \\
};

\node[below=9mm of iind] {feature index $i = 1,\dots,N$};


\draw[->,thick]
    (eta-1-2.south) .. controls +(0,-0.8) and +(-1,0) .. (lab2.west);

\draw[->,thick]
    (eta-1-4.south) .. controls +(0,-0.8) and +(-1,0) .. (lab4.west);

\draw[->,thick]
    (eta-1-5.south) .. controls +(0,-0.8) and +(-1,0) .. (lab5.west);

\node[align=left] at (-0.5,-3.2)
  {$\eta^\alpha_\mu = 1 \;\Rightarrow\;$ concept $\alpha$\\uses stroke $\mu$ with pattern $\xi^\mu$};

\end{tikzpicture}
    \caption{Visual representation of the feature-level representation of concept $\alpha$ (the letter F). The concept is formed by the pixelwise OR ($\bigvee$) of three sparse stroke patterns: $\xi^1$ (vertical), $\xi^2$ (top), and $\xi^3$ (middle), such that $\eta^{(\alpha)} = \bigvee_{\nu \in S_\alpha} \xi^\nu$ with $L=3$.
Illustration of the retrieval process. The overlap $O_\mu$ is calculated between the stored concept $\eta^{(\alpha)}$ and a test stroke $\xi^\mu$ by summing the shared active pixels. The system activates the stroke if $O_\mu$ exceeds the threshold $\theta$.}
    \label{fig:overlap_calc}
\end{figure}

A \emph{concept} $\alpha$ is specified by a subset of strokes
\[
S_\alpha \subset \{1,\dots,M\}\qquad \text{ with }|S_\alpha|=L,
\]
where the concept size $L$ is initially fixed (we relax this condition in Section \ref{sec:var}).  The connection between concepts, strokes and patterns is illustrated in Figure \ref{fig:overlap_calc}. 

The second (concept) layer stores $P$ concept patterns
\[
\eta^\alpha\in\{0,1\}^M \quad\text{for }  \alpha=1,\dots,P, \qquad  \text{where } \eta^\alpha_\mu=\mathbf 1\{\mu\in S_\alpha\}\quad (\mu=1,\dots,M),
\]
and receives the stroke-layer output $x$ as input. Note that 
$$
\sum_{\mu=1}^M \eta_\mu^\alpha=\|\eta^\alpha\|_1=L \qquad \text{for all }  \alpha=1,\dots,P.
$$

On the concept level, we decode a simple winner-take-all
rule based on concept scores
\begin{eqnarray}
	\label{eq:concept-score}
	H_\beta(x)&:=&\sum_{\mu=1}^M \eta^\beta_\mu\,x_\mu = |S_\beta\cap \mathrm{supp}(x)|,
	\qquad \text{and}\\
	\label{eq:argmax}
	\widehat\alpha(x)&=&\arg\max_{\beta\le P} H_\beta(x).
\end{eqnarray}
If there are several concepts maximizing $H_\beta(x)$ we take an arbitrary tie-breaking rule to choose one of them. 

The decoding therefore selects the concept whose associated stroke set has the largest overlap with the observed activation pattern, effectively aggregating evidence across multiple strokes.
In the ideal case $x=\eta^\alpha$, we have $H_\alpha(x)=L$ and $H_\beta(x)=|S_\beta\cap
S_\alpha|$.

\subsection{Heuristic picture and error decomposition}

There are two types of errors that can occur in the stroke layer: 
\begin{itemize}
	\item \emph{false negatives:} strokes in $S_\alpha$ that are inactive, and
	\item \emph{false positives:} strokes outside $S_\alpha$ that are spuriously activated.
\end{itemize}
For a fixed concept $\alpha$, define the aggregate error counts
\[
F_-:=F_-(\alpha):=\sum_{\mu\in S_\alpha}{(1-x_\mu)}
\quad \text{ as well as } \quad
F_+:=F_+(\alpha):=\sum_{\mu\notin S_\alpha} x_\mu.
\]
Successful concept retrieval can be guaranteed under an event that controls these
aggregate errors rather than enforcing exact recovery of all strokes. 

Concretely, for
parameters $\delta\in(0,1)$ and $\rho\in\mathbb N_0$, we define the ``good event'' 
\begin{equation}
	\label{eq:good-stroke-event-def}
	\mathcal G_{\alpha}(\delta,\rho)
	:=
	\Bigl\{\sum_{\mu\in S_\alpha} x_\mu \ge (1-\delta)L \Bigr\}
	\cap
	\Bigl\{\sum_{\mu\notin S_\alpha} x_\mu \le \rho\Bigr\}.
\end{equation}
The first condition bounds missed true strokes, the second bounds the number of spurious
strokes. As we show below, for sparse strokes the false-positive probability is
super-polynomially small, while false negatives provide the dominant restriction in exact-recovery regimes. 
A key idea in our approach is that allowing robust recovery through $\mathcal G_\alpha(\delta,\rho)$
replaces union bounds over $M$ events by concentration of aggregate counts and can
significantly relax admissible scalings. 

Note that, of course, instead of the notation $(1-\delta)L$, one could also fix an integer in $\{1, \dots, L\}$. However, in Section \ref{sec:var}, we will consider variable concept sizes, and in this case, the notation $(1-\delta)L$ is more convenient.

%
%
\section{Results for concepts with fixed size}
\label{sec:fixed}

In this section we analyse hierarchical retrieval for concepts of a fixed size
$L$. Our main point is that successful concept retrieval does not require exact
recovery of all constituent strokes at the intermediate layer. Instead, it is
enough that the stroke layer produces sufficiently many true positives and only
few false positives in aggregate, provided competing concepts do not overlap too
strongly with the target.

We begin by fixing notation and formulating a deterministic retrieval criterion.
This criterion separates the analysis into two parts: control of the concept
geometry (through overlaps between concept sets) and control of the stroke-layer
errors (through aggregate false negatives and false positives). The resulting
probabilistic bounds are then combined in Theorem~\ref{thm:retrieve-concepts},
which is the main result of this section.

\subsection{Setup and notation}

Fix $L\in\mathbb N$. For each concept $\alpha\in\{1,\dots,P\}$, let
\[
S_\alpha \subseteq \{1,\dots,M\}, \qquad |S_\alpha|=L,
\]
denote its associated set of strokes, and let $\eta^{(\alpha)}$ be the clean
concept cue defined in \eqref{eq:concept-or}. Recall that we assume the $M$ stroke patterns
\[
\xi^\mu = (\xi_i^\mu)_{i=1}^{N_f} \in \{0,1\}^{N_f}, \qquad \mu=1,\dots,M,
\]
to be i.i.d.\ sparse with success probability 
$$p=\frac{\log N_f}{N_f}.$$
 From this cue we generate the
stroke-layer output
\[
x=x(\eta^{(\alpha)})\in\{0,1\}^M
\]
via the threshold rule \eqref{eq:stroke-threshold-rule} with
$\theta := \kappa \log N_f$ for some constant $0<\kappa<1$.

 The concept score is
given by
\[
H_\beta(x)=\sum_{\mu=1}^M \eta^\beta_\mu x_\mu
= |S_\beta\cap \supp(x)|.
\]
To measure stroke-layer errors relative to the target concept $\alpha$, we define
\[
F_-^{(\alpha)}:=\sum_{\mu\in S_\alpha}(1-x_\mu)
\qquad \text{and } \quad
F_+^{(\alpha)}:=\sum_{\mu\notin S_\alpha}x_\mu,
\]
that is, the numbers of false negatives and false positives, respectively.
For parameters $\delta\in(0,1)$ and $\rho\in\mathbb N_0$, we consider the good event
\[
\mathcal G_\alpha(\delta,\rho)
:=
\{F_-^{(\alpha)}\le \delta L\}\cap \{F_+^{(\alpha)}\le \rho\}.
\]
Finally, we write
\[
t_*^{(\alpha)}:=\max_{\beta\neq\alpha}|S_\beta\cap S_\alpha|
\]
for the maximal overlap between the target concept and its competitors.

The key relaxation in our analysis is that we do not insist on exact recovery of
all strokes. Instead, we work with the aggregate event
$\mathcal G_\alpha(\delta,\rho)$, which only requires that most target strokes
are activated and that the number of spurious activations remains small.
The next proposition shows that this is already sufficient for correct
winner-take-all decoding, as long as competing concepts do not overlap too
strongly with the target.

\subsection{Robust concept retrieval}

\begin{proposition}[Deterministic score separation]
	\label{lem:score-separation}
	Fix a target concept $\alpha$. On the event $\mathcal G_\alpha(\delta,\rho)$, one has
	for every competitor $\beta\neq\alpha$,
	\begin{equation}\label{eq:score_sep}
	H_\alpha(x)-H_\beta(x)
	\ge ((1-\delta)L-t_*^{(\alpha)})-\rho.
	\end{equation}
	In particular, if
	\[
	t_*^{(\alpha)}+\rho < (1-\delta)L,
	\]
	then the winner-take-all decoder selects the correct concept:
	\[
	\widehat\alpha(x)=\alpha.
	\]
\end{proposition}

Proposition~\ref{lem:score-separation} reduces successful retrieval to two
separate requirements: first, the target concept must be sufficiently well
represented at the stroke layer, as encoded by the good event
$\mathcal G_\alpha(\delta,\rho)$; second, no competing concept may overlap too
strongly with the target. In the i.i.d.\ concept model, the second requirement
can be controlled by a simple combinatorial overlap bound, while the first is
handled by probabilistic estimates for false negatives and false positives.
Combining these ingredients yields the following theorem, which is the main result of this section. 
It shows that concept retrieval remains reliable even when the intermediate stroke representation is imperfect, as long as the aggregate error level is controlled and the concept family has sufficient separation.

\begin{theorem}[Retrieving concepts]
\label{thm:retrieve-concepts}
In the above setting, assume that there exists an integer $t\in\{0,\dots,L-1\}$ such that the \emph{margin condition}
	\begin{equation}
		\label{eq:thm-margin}
		t+\rho < (1-\delta)L
	\end{equation}
	holds. 
	Then, for every fixed $\alpha$,
	\begin{equation}
		\label{eq:thm-success-prob-lower}
		\mathbb P\bigl(\widehat\alpha(x)=\alpha\bigr)
		\ \ge\
		1
		-\mathbb P\bigl(t_*^{(\alpha)}\ge t+1\bigr)
		-\mathbb P\bigl(\mathcal G_\alpha(\delta,\rho)^c\bigr).
	\end{equation}
	Moreover, in the i.i.d.\ concept model one has the explicit bound
	\begin{equation}
		\label{eq:thm-overlap-bound}
		\mathbb P\bigl(t_*^{(\alpha)}\ge t+1\bigr)
		\ \le\
		P\,\binom{L}{t+1}\Bigl(\frac{L}{M}\Bigr)^{t+1}.
	\end{equation}
	In particular, if $P=P(N_f)$ and $M=M(N_f)$ are such that
	\begin{equation}
		\label{eq:thm-capacity-conditions}
		P\,\binom{L}{t+1}\Bigl(\frac{L}{M}\Bigr)^{t+1} \xrightarrow[N_f\to\infty]{} 0
		\quad
		\text{and}
		\quad
		\max_{\alpha\le P}\mathbb P\bigl(\mathcal G_\alpha(\delta,\rho)^c\bigr)
		\xrightarrow[N_f\to\infty]{} 0,
	\end{equation}
	then the hierarchical retrieval \emph{succeeds with vanishing error} in the sense that
	\[
	\max_{\alpha\le P}\mathbb P\bigl(\widehat\alpha(x)\neq \alpha\bigr)
	\xrightarrow[N_f\to\infty]{}0 .
	\]
\end{theorem}

The theorem makes precise that exact stroke recovery is not needed for successful
concept decoding. What matters is only that the target concept retains enough of
its true strokes, that few spurious strokes are activated, and that competitors
do not overlap too strongly with the target.
As a simple consequence, one obtains the following polynomial growth regimes for
the numbers of strokes and concepts.

%
%
\begin{corollary}
	\label{cor:poly-gamma}
	In the above setting,
	choose an
	integer $t\in\{0,\dots,L-1\}$ such that the margin condition \eqref{eq:thm-margin}
	holds. Let $P=P(N_f)=N_f^r$ for some $r>0$, and suppose $M=M(N_f)=N_f^\gamma$ for
	some $\gamma>0$.
	
	\begin{itemize}
		\item[\textup{(i)}] \emph{(Fixed target concept)} If $\gamma>r/(t+1)$ and
		$\max_{\alpha\le P}\mathbb P(\mathcal G_\alpha(\delta,\rho)^c)\to 0$, then
		\[
		\mathbb P(\widehat\alpha(x)\neq\alpha)\xrightarrow[N_f\to\infty]{}0
		\qquad\text{for every fixed }\alpha.
		\]
		
		\item[\textup{(ii)}] \emph{(Uniform over all concepts)} If $\gamma>2r/(t+1)$ and
		$\max_{\alpha\le P}\mathbb P(\mathcal G_\alpha(\delta,\rho)^c)\to 0$, then
		\[
	\mathbb P(\exists \alpha \le P: \widehat\alpha(x)\neq\alpha)\xrightarrow[N_f\to\infty]{}0.
		\]
	\end{itemize}
\end{corollary}

Moreover, to actually apply Theorem~\ref{thm:retrieve-concepts}, one needs to control the probability of the good event $\mathcal G_\alpha(\delta,\rho)$, which encodes the stroke-layer errors.

%
%
{\begin{proposition}[Control of the good event]
		\label{lem:G-fixed}
		Fix $\delta\in(0,1)$ and $\rho\in\mathbb N_0$, and let $m:=\lfloor \delta L\rfloor+1$.
		Define
		\[
		q_-^{\mathrm{true}}
		:=
		\mathbb P\bigl(|\xi^\mu|_1<\theta\bigr),
		\qquad \mu\in S_\alpha,
		\]
		and suppose that there exists a deterministic quantity $\overline q_+\ge 0$ such that
		\[
		q_+:=\max_{\mu\notin S_\alpha}\mathbb P\bigl(x_\mu=1\mid \eta^{(\alpha)}\bigr)
		\le \overline q_+
		\qquad\text{almost surely.}
		\]
		Then
		\[
		\mathbb P\bigl(\mathcal G_\alpha(\delta,\rho)^c\bigr)
		\le
		\binom{L}{m}\bigl(q_-^{\mathrm{true}}\bigr)^m
		+
		\frac{\bigl((M-L)\overline q_+\bigr)^{\rho+1}}{(\rho+1)!}.
		\]
		In particular, if
		\[
		\binom{L}{m}\bigl(q_-^{\mathrm{true}}\bigr)^m \xrightarrow[N_f\to\infty]{}0
		\qquad\text{and}\qquad
		(M-L)\overline q_+ \xrightarrow[N_f\to\infty]{}0,
		\]
		then
		\[
		\mathbb P\bigl(\mathcal G_\alpha(\delta,\rho)^c\bigr)\xrightarrow[N_f\to\infty]{}0.
		\]
\end{proposition}}

In our model this results in: 

\begin{proposition}
	\label{prop:G-fixed-sparse}
	Fix $\delta\in(0,1)$ and $\rho\in\mathbb N_0$, and let
	$m:=\lfloor \delta L\rfloor+1$.
	Under our assumptions let $\alpha$ be fixed.
	Then there exist constants $a=a(\kappa)>0$, $C=C(L)>0$, and $c=c(\kappa,L,C)>0$
	such that
	\[
	\mathbb P\bigl(\mathcal G_\alpha(\delta,\rho)^c\bigr)
	\le
	L\,N_f^{-a}
	+
	\mathbb P\bigl(\|\eta^{(\alpha)}\|_1>C\log N_f\bigr)
	+
	\frac{\bigl((M-L)e^{-c(\log N_f)^2}\bigr)^{\rho+1}}{(\rho+1)!}.
	\]
	In particular, if $M=M(N_f)$ grows at most polynomially in $N_f$, then
	\[
	\mathbb P\bigl(\mathcal G_\alpha(\delta,\rho)^c\bigr)\xrightarrow[N_f\to\infty]{}0.
	\]
\end{proposition}

\subsection{Exact stroke recovery as a benchmark}\label{subsec:exact_recovery}

For comparison, we now record a stronger statement in which the stroke layer is
required to recover the active stroke set exactly from a clean concept cue.
This is a substantially more demanding requirement than robust concept retrieval,
since it enforces uniform control over the entire stroke dictionary. As a
consequence, the admissible growth of $M$ is more restrictive.

\begin{proposition}
	\label{prop:strict-exact}
	Fix $\kappa\in(0,1)$ and set $\theta=\kappa\log N_f$.
	Under our assumptions for a concept $\alpha$ let the clean cue
	$\eta^{(\alpha)}$ be given by \eqref{eq:concept-or}, and define stroke activations by
	\eqref{eq:stroke-threshold-rule}.
	Then, there exists $a=a(\kappa)>0$ such that, if $M\le N_f^{a-\varepsilon}$ for some
	$\varepsilon>0$, the stroke layer recovers the active stroke set exactly from the clean
	cue, i.e.
	\[
	x_\mu(\eta^{(\alpha)})=\mathbf 1\{\mu\in S_\alpha\}\qquad\text{for all }\mu\le M,
	\]
	with probability $1-o(1)$ as $N_f\to\infty$.
\end{proposition}

\begin{remark}[Exact versus robust recovery]
		Proposition~\ref{prop:strict-exact} requires exact recovery of the active
	stroke set and, in particular, uniform control over all stored strokes
	$\{\xi^\mu\}_{\mu\le M}$. This yields an explicit polynomial sufficient
	condition on the dictionary size, namely
	$
	M \le N_f^{a(\kappa)-\varepsilon}.
	$
	By contrast, the robust retrieval theory only requires aggregate control of
	the stroke-layer errors together with the overlap condition from
	Theorem~\ref{thm:retrieve-concepts}. In particular, Corollary~\ref{cor:poly-gamma}
	shows that, if $P=N_f^r$ and $M=N_f^\gamma$, then robust retrieval of a fixed
	target concept is compatible with every exponent
	$
	\gamma>\nicefrac{r}{t+1},
	$
	while uniform retrieval over all concepts is compatible with every exponent
	$
	\gamma>\nicefrac{2r}{t+1},
	$
	provided the good event holds with high probability.
	Hence, whenever
	$
	\gamma>\max\big\{a(\kappa),\nicefrac{r}{t+1}\big\},
	$
	the regime $M=N_f^\gamma$ lies beyond the exact-recovery sufficient
	condition of Proposition~\ref{prop:strict-exact}, but is still admissible
	for robust retrieval of a fixed target concept. Likewise, every
	$
	\gamma>\max\big\{a(\kappa),\nicefrac{2r}{t+1}\big\}
	$
	lies beyond the exact-recovery sufficient condition and remains admissible
	for uniform robust retrieval over all concepts.

\end{remark}

The next remark explains why exact recovery over the full stored dictionary is
stronger than what the retrieval task actually requires.
 
\begin{remark}[Used versus stored strokes]
	\label{rem:used-strokes}
	Proposition~\ref{prop:strict-exact} enforces exact recovery uniformly over the
	entire stored stroke dictionary $\{\xi^\mu\}_{\mu\le M}$. For concept retrieval,
	however, this is stronger than necessary: only strokes that belong to at least
	one concept can influence the winner-take-all decision. Strokes that are never
	used in any concept are irrelevant for retrieval and therefore need not be
	recovered exactly.
	
	This suggests a refined exact-recovery statement in which uniform control is
	required only on the set of used strokes. The corresponding bound then depends
	on the number of used strokes rather than on the full dictionary size $M$.
\end{remark}

To make this precise, let
\[
\mathcal U := \bigcup_{\alpha=1}^P S_\alpha \subseteq \{1,\dots,M\}
\]
denote the set of used strokes, and write $U:=|\mathcal U|$ for its cardinality.
We then obtain the following refinement of Proposition~\ref{prop:strict-exact}.

\begin{proposition}[Exact recovery restricted to used strokes]
	\label{prop:strict-exact-used}
 In the above setting there exists $a=a(\kappa)>0$ such that, if $U\le N_f^{a-\varepsilon}$ for some
	$\varepsilon>0$, the stroke layer recovers every concept's stroke set exactly from its
	clean cue in the sense that, with probability $1-o(1)$,
	\[
	x_\mu(\eta^{(\alpha)})=\mathbf 1\{\mu\in S_\alpha\}
	\qquad\text{for all }\alpha\le P\text{ and all }\mu\in\mathcal U.
	\]
	Moreover, for each fixed $\alpha$, the same conclusion holds without the union over
	$\alpha$ (i.e.\ only for the fixed cue $\eta^{(\alpha)}$).
	
	Since $U:=|\mathcal U|\le PL$, the strict exact-recovery conclusion holds in particular
	whenever
	\[
	PL \le N_f^{a-\varepsilon}.
	\]
\end{proposition}

Even this refinement of considering used strokes only still imposes an exact-recovery requirement at the stroke
level. The robust retrieval results above show that such exact control is in
fact unnecessary for successful concept decoding.

\section{Proofs for results in Section~\ref{sec:fixed}}
\label{sec:proofs}

This section contains the proofs of the results stated in Section~\ref{sec:fixed}.
We follow the same order as there. We begin with the robust concept-retrieval
results, whose logic separates into a deterministic score-separation argument,
a probabilistic overlap bound, and a control of the good event
$\mathcal G_\alpha(\delta,\rho)$. We then turn to the exact stroke-recovery
statements, first for the full dictionary and then for the refinement to used
strokes only.

\subsection{Robust concept retrieval}
We start with the deterministic part of the argument. The first proposition shows
that successful decoding follows once the target concept is sufficiently well
represented at the stroke layer and competing concepts have limited overlap with
the target.

\begin{proof}[Proof of Proposition~\ref{lem:score-separation}]
	On the event $\mathcal G_\alpha(\delta,\rho)$, we have
	\[
	F_-^{(\alpha)}\le \delta L
	\qquad\text{and}\qquad
	F_+^{(\alpha)}\le \rho.
	\]
	Hence
	\[
	H_\alpha(x)
	=
	\sum_{\mu\in S_\alpha}x_\mu
	=
	L-F_-^{(\alpha)}
	\ge (1-\delta)L.
	\]
	On the other hand, for every $\beta\neq\alpha$,
	\[
	H_\beta(x)
	=
	\sum_{\mu\in S_\beta\cap S_\alpha}x_\mu
	+
	\sum_{\mu\in S_\beta\setminus S_\alpha}x_\mu
	\le
	|S_\beta\cap S_\alpha|
	+
	\sum_{\mu\notin S_\alpha}x_\mu
	\le
	t_*^{(\alpha)}+\rho.
	\]
	Therefore,
	\[
	H_\alpha(x)-H_\beta(x)
	\ge
	(1-\delta)L-t_*^{(\alpha)}-\rho,
	\]
	which proves \eqref{eq:score_sep}. In particular, if
	\[
	t_*^{(\alpha)}+\rho<(1-\delta)L,
	\]
	then $H_\alpha(x)>H_\beta(x)$ for every $\beta\neq\alpha$, and hence
	$\widehat\alpha(x)=\alpha$.
\end{proof}

We next control the geometric contribution to the retrieval error, namely the
maximal overlap between the target concept and its competitors. In the i.i.d.\
concept model, this is given by the following combinatorial bound.

\begin{lemma}[Overlap bound]
	\label{lem:overlap}
	Fix $\alpha$ and let $S_\alpha,S_\beta$ be independent uniform $L$-subsets of
	$\{1,\dots,M\}$. For any $t\in\{0,\dots,L-1\}$,
	\[
	\mathbb P\bigl(|S_\beta\cap S_\alpha|\ge t+1\bigr)
	\le
	\binom{L}{t+1}\left(\frac{L}{M}\right)^{t+1}.
	\]
	Consequently,
	\[
	\mathbb P\Bigl(\max_{\beta\neq\alpha}|S_\beta\cap S_\alpha|\ge t+1\Bigr)
	\le
	P\,\binom{L}{t+1}\left(\frac{L}{M}\right)^{t+1}.
	\]
\end{lemma}

\begin{proof}
	Let $\beta\neq\alpha$ and write
	\[
	X_\beta:=|S_\beta\cap S_\alpha|.\]
	If $X_\beta\ge t+1$, then there exists a subset
	$T\subseteq S_\alpha$ with $|T|=t+1$ such that $T\subseteq S_\beta$.
	Hence
	\[
	\{X_\beta\ge t+1\}
	\subseteq
	\bigcup_{\substack{T\subseteq S_\alpha\\|T|=t+1}}
	\{T\subseteq S_\beta\}.
	\]
	By a union bound,
	\[
	\mathbb P(X_\beta\ge t+1)
	\le
	\binom{L}{t+1}\,\mathbb P(T\subseteq S_\beta),
	\]
	where $T$ is any fixed $(t+1)$-subset of $S_\alpha$.
	Since $S_\beta$ is uniformly distributed among all $L$-subsets of
	$\{1,\dots,M\}$,
	\[
	\mathbb P(T\subseteq S_\beta)
	=
	\frac{\binom{M-(t+1)}{L-(t+1)}}{\binom{M}{L}}
	=
	\prod_{k=0}^{t}\frac{L-k}{M-k}
	\le
	\left(\frac{L}{M}\right)^{t+1}.
	\]
	Therefore
	\[
	\mathbb P(X_\beta\ge t+1)
	\le
	\binom{L}{t+1}\left(\frac{L}{M}\right)^{t+1}.
	\]
	A second union bound over $\beta\neq\alpha$ gives
	\[
	\mathbb P\Bigl(\max_{\beta\neq\alpha}|S_\beta\cap S_\alpha|\ge t+1\Bigr)
	\le
	P\,\binom{L}{t+1}\left(\frac{L}{M}\right)^{t+1}.
	\]	
\end{proof}

The main theorem of the previous section is then obtained by combining this deterministic separation
criterion with the overlap estimate for the random concept family.

\begin{proof}[Proof of Theorem~\ref{thm:retrieve-concepts}]
	By Proposition~\ref{lem:score-separation}, retrieval succeeds whenever
	$\mathcal G_\alpha(\delta,\rho)$ holds and $t_*^{(\alpha)}\le t$.
	Hence,
	\[
	\mathbb P(\widehat\alpha(x)\neq\alpha)
	\le
	\mathbb P(t_*^{(\alpha)}\ge t+1)
	+
	\mathbb P(\mathcal G_\alpha(\delta,\rho)^c).
	\]
	The overlap term is bounded by Lemma~\ref{lem:overlap}, which yields
	\[
	\mathbb P(t_*^{(\alpha)}\ge t+1)
	\le
	P\binom{L}{t+1}\left(\frac{L}{M}\right)^{t+1}.
	\]
	This proves \eqref{eq:thm-success-prob-lower} and \eqref{eq:thm-overlap-bound}.
	If, in addition, the two conditions in \eqref{eq:thm-capacity-conditions} hold,
	then both terms on the right-hand side vanish uniformly in $\alpha$, and therefore
	\[
	\max_{\alpha\le P}\mathbb P(\widehat\alpha(x)\neq\alpha)\xrightarrow[N_f\to\infty]{}0.
	\]
\end{proof}

We now prove Corollary \ref{cor:poly-gamma}.

\begin{proof}[Proof of Corollary~\ref{cor:poly-gamma}]
	By Theorem~\ref{thm:retrieve-concepts},
	\[
	\mathbb P(\widehat\alpha(x)\neq\alpha)
	\le
	P\binom{L}{t+1}\left(\frac{L}{M}\right)^{t+1}
	+
	\mathbb P(\mathcal G_\alpha(\delta,\rho)^c).
	\]
	Since $L$ and $t$ are fixed, the first term is of order
	\[
	N_f^r \cdot N_f^{-\gamma(t+1)}
	=
	N_f^{\,r-\gamma(t+1)}.
	\]
	Thus it tends to zero whenever $\gamma>r/(t+1)$, which proves (i).
	
	For the uniform statement, take the maximum over $\alpha\le P$ and use again
	Theorem~\ref{thm:retrieve-concepts}. This introduces an additional factor $P$,
	so the overlap term is of order
	\[
	P^2\left(\frac{L}{M}\right)^{t+1}
	\asymp
	N_f^{\,2r-\gamma(t+1)}.
	\]
	Hence it tends to zero whenever $\gamma>2r/(t+1)$, proving (ii).
\end{proof}

\subsection{Proofs for the control of the good event}
It remains to control the second term in
Theorem~\ref{thm:retrieve-concepts}, namely the probability that the good event
$\mathcal G_\alpha(\delta,\rho)$ fails. We first formulate a convenient bound in
terms of the envelope probabilities $q_-$ and $q_+$, and then derive the
auxiliary estimates needed to verify it.

\begin{proof}[Proof of Proposition~\ref{lem:G-fixed}]
A union bound gives
		\[
		\mathbb P\bigl(\mathcal G_\alpha(\delta,\rho)^c\bigr)
		\le
		\mathbb P\bigl(F_-^{(\alpha)}\ge m\bigr)
		+
		\mathbb P\bigl(F_+^{(\alpha)}>\rho\bigr),
		\]
		where $m=\lfloor \delta L\rfloor+1$.
		
		We first consider the false negatives. For $\mu\in S_\alpha$, the clean cue
		satisfies $\eta_i^{(\alpha)}\ge \xi_i^\mu$ for every $i$, hence $\eta_i^{(\alpha)}\xi_i^\mu=\xi_i^\mu$
		for all $i$.
		Therefore
		\[
		O_\mu
		=
		\sum_{i=1}^{N_f}\eta_i^{(\alpha)}\xi_i^\mu
		=
		\sum_{i=1}^{N_f}\xi_i^\mu
		=
		|\xi^\mu|_1.
		\]
		It follows that $x_\mu=0$ if and only if 
        $|\xi^\mu|_1<\theta$.
		Thus, for $\mu\in S_\alpha$,
		\[
		Y_\mu:=\mathbf 1\{x_\mu=0\}
		=
		\mathbf 1\{|\xi^\mu|_1<\theta\}.
		\]
		Since the strokes $(\xi^\mu)_{\mu\in S_\alpha}$ are i.i.d., the variables
		$(Y_\mu)_{\mu\in S_\alpha}$ are i.i.d.\ Bernoulli with parameter
		$q_-^{\mathrm{true}}$. Hence
		\[
		F_-^{(\alpha)}=\sum_{\mu\in S_\alpha}Y_\mu
		\sim \mathrm{Bin}(L,q_-^{\mathrm{true}}).
		\]
		Therefore
		\[
		\mathbb P\bigl(F_-^{(\alpha)}\ge m\bigr)
		\le
		\binom{L}{m}\bigl(q_-^{\mathrm{true}}\bigr)^m.
		\]
		
		We now turn to the false positives. Conditional on $\eta^{(\alpha)}$, the
		variables $(x_\mu)_{\mu\notin S_\alpha}$ are independent Bernoulli random
		variables, and by assumption each of them has success probability at most
		$\overline q_+$. Hence, conditional on $\eta^{(\alpha)}$,
		\[
		F_+^{(\alpha)}
		\preceq_{\mathrm{st}}
		\mathrm{Bin}(M-L,\overline q_+).
		\]
		Thus, using the falling-factorial moment bound for binomial tails, we obtain
		\[
		\mathbb P\bigl(F_+^{(\alpha)}>\rho \mid \eta^{(\alpha)}\bigr)
		\le
		\frac{\bigl((M-L)\overline q_+\bigr)^{\rho+1}}{(\rho+1)!}.
		\]
		Indeed, if $B\sim\mathrm{Bin}(n,q)$, then
		\[
		\mathbb P(B>\rho)
		=
		\mathbb P(B\ge \rho+1)
		=
		\mathbb P\bigl((B)_{\rho+1}\ge (\rho+1)!\bigr)
		\le
		\frac{\mathbb E[(B)_{\rho+1}]}{(\rho+1)!}
		\le
		\frac{(nq)^{\rho+1}}{(\rho+1)!},
		\]
		where $(B)_{\rho+1}=B(B-1)\cdots(B-\rho)$.
		
		Taking expectations yields
		\[
		\mathbb P\bigl(F_+^{(\alpha)}>\rho\bigr)
		\le
		\frac{\bigl((M-L)\overline q_+\bigr)^{\rho+1}}{(\rho+1)!}.
		\]
		
		Combining the two estimates proves the first claim. The final implication is immediate.
\end{proof}

To estimate the false-positive envelope probability $q_+$ in the sparse stroke
model, we use the following binomial tail bound.
This bound could actually also be obtained by Chernoff's inequality. We prove it to keep the paper self-contained. 
\begin{lemma}
	\label{lem:binom-tail}
	Let $X\sim\mathrm{Bin}(T_{\mathrm{ bin}},p_{\mathrm{bin}})$ with $T_{\mathrm {bin}}\in\mathbb N$, $p_{\mathrm{bin}}\in(0,1)$.
    Then for every integer $k>Tp$,
	\[
	\mathbb P(X\ge k)\le \left(\frac{eT_{\mathrm{bin}}p_{\mathrm{bin}}}{k}\right)^k .
	\]
\end{lemma}
\begin{proof}
	We have $
	\mathbb P(X\ge k)
	= \sum_{j=k}^{T_{\mathrm{bin}}} \binom{T_{\mathrm{bin}}}{j} p_{\mathrm{bin}}^j (1-p)^{T_{\mathrm{bin}}-j}$.
	Now, for every $j\ge k$,
	\[
	\binom{T_{\mathrm{bin}}}{j}\le \binom{T_{\mathrm{bin}}}{k}\binom{T_{\mathrm{bin}}-k}{j-k}.
	\]
	Therefore,
	\[
	\mathbb P(X\ge k)
	\le \binom{T_{\mathrm{bin}}}{k}p_{\mathrm{bin}}^k
	\sum_{j=k}^{T_{\mathrm{bin}}} \binom{T_{\mathrm{bin}}-k}{j-k} p_{\mathrm{bin}}^{\,j-k}(1-p)^{T_{\mathrm{bin}}-j}.
	\]
	With the change of variables $m=j-k$, the sum becomes $
	\sum_{m=0}^{T_{\mathrm{bin}}-k}\binom{T_{\mathrm{bin}}-k}{m}p_{\mathrm{bin}}^m(1-p)^{(T_{\mathrm{bin}}-k)-m}=1.$
	Hence
	\[
	\mathbb P(X\ge k)\le \binom{T_{\mathrm{bin}}}{k}p_{\mathrm{bin}}^k.
	\]
	Finally, using the  bound
	$\binom{T_{\mathrm{bin}}}{k}\le \left(\frac{eT_{\mathrm{bin}}}{k}\right)^k$,
	we obtain the claim.
\end{proof}

{We now estimate the one-stroke false-positive probability $q_+$ in the sparse
	stroke model. The next lemma shows that a stroke outside the target concept
	fires only with super-polynomially small probability. This is the key input for
	controlling the total number of spurious activations in the stroke layer.}

To prepare the next step we need to bound the typical size of a clean cue:

\begin{lemma}
	\label{lem:cue-size-upper}
	Fix a concept $\alpha$ with $|S_\alpha|=L$, and let
	$T:=\|\eta^{(\alpha)}\|_1$
	denote the number of active coordinates in the clean cue
	$\eta^{(\alpha)}$. Then, under our standing assumptions for every $C>eL$  and all $N_f$ large enough,
	\[
	\mathbb P\bigl(T> C\log N_f\bigr)
	\le
	\exp\bigl(-c(C,L)\log N_f\bigr)
	= N_f^{-c(C,L)}
	\]
   In particular, for such $C$ we have 
	\[
\mathbb P\bigl(T\le C\log N_f\bigr)\xrightarrow[N_f\to\infty]{}1.
\]
\end{lemma}

\begin{proof}
	This is just a typical exponential estimate: 
	For each coordinate $i\in\{1,\dots,N_f\}$,
	$\eta_i^{(\alpha)}
	= \mathbf 1\Bigl\{\exists\,\nu\in S_\alpha:\ \xi_i^\nu=1\Bigr\}$.
	Since the stroke patterns are independent and $|S_\alpha|=L$, we have
	\[
	\mathbb P(\eta_i^{(\alpha)}=1)
	=
	1-(1-p)^L
	=:\pi_L.
	\]
	Moreover, the families $\{(\xi_i^\nu)_{\nu\in S_\alpha}\}_{i=1}^{N_f}$ are
	independent across $i$, hence the coordinates
	$\eta_i^{(\alpha)}$ are i.i.d.\ Bernoulli$(\pi_L)$. Therefore
	\[
	T=\sum_{i=1}^{N_f}\eta_i^{(\alpha)}\sim \mathrm{Bin}(N_f,\pi_L).
	\]
	Using Bernoulli's inequality, $\pi_L=1-(1-p)^L\le Lp$,
	and therefore $\mathbb E[T]\le LN_fp=L\log N_f$.
	
	Now, fix any $C>L$. Then for all sufficiently large $N_f$, $C\log N_f > \mathbb E[T]$.
	Applying Lemma~\ref{lem:binom-tail} with
	$X=T, T_{\mathrm{bin}}=N_f,p_{\mathrm{bin}}=\pi_L$,
	and $k=\lfloor C\log N_f\rfloor$,
	we obtain
	\[
	\mathbb P(T\ge C\log N_f)
	\le
	\left(\frac{e\,N_f\pi_L}{C\log N_f}\right)^{C\log N_f}.
	\]
	Since $N_f\pi_L\le L\log N_f$, this gives
	$
	\mathbb P(T\ge C\log N_f)
	\le
	\left(\frac{eL}{C}\right)^{C\log N_f}.
	$
	If $C>eL$, then $\frac{eL}{C}<1$, and hence
	\[
	\mathbb P(T\ge C\log N_f)
	\le
	\exp\!\left(-c(C,L)\log N_f\right)
	=
	N_f^{-c(C,L)}
	\]
	for some $c(C,L)>0$.
	In particular,
	\[
	\mathbb P(T\le C\log N_f)\to 1.
	\]
\end{proof}
%
%
\begin{lemma}
	\label{lem:false-positives}
	Fix $\alpha$ and assume $T:=\|\eta^{(\alpha)}\|_1 \le C \log N_f$ for a $C >eL$. For $0 <\kappa <1$  set $\theta=\kappa\log N_f$. Then for a false stroke
	$\mu\notin S_\alpha$,
	\[
	q_+ := \mathbb P(x_\mu=1\mid \eta^{(\alpha)})
	\le
	\exp\bigl(-c(\log N_f)^2\bigr)
	\]
	for some $c=c(\kappa,L)>0$ and all $N_f$ large enough.
	In particular, $q_+$ is super-polynomially small in $N_f$.
\end{lemma}

\begin{proof}
	Fix $\alpha$ and a ``false stroke'' $\mu\notin S_\alpha$. Recall that 
	by the assumptions, the stroke patterns $\{\xi^\nu\}_{\nu=1}^M$ are independent
	across $\nu$, and in particular $\xi^\mu$ is independent of
	$\{\xi^\nu:\nu\in S_\alpha\}$ and thus independent of $\eta^{(\alpha)}$.
	
	Now, conditional on
	$\eta^{(\alpha)}$ the coordinates $(\xi_i^\mu)_{i=1}^{N_f}$ remain i.i.d.\
	$\mathrm{Bernoulli}(p)$ and are independent of $\eta^{(\alpha)}$. 
	
	Therefore, conditional on $\eta^{(\alpha)}$,
	\[
	O_\mu=\sum_{i=1}^{N_f}\eta^{(\alpha)}_i\,\xi_i^\mu
	\ \Big|\ \eta^{(\alpha)}
	\ \sim\ \mathrm{Bin}(T,p).
	\]
	
	Set $\mathbb E[O_\mu\mid \eta^{(\alpha)}]=Tp.$
	By a Chernoff bound for binomial random variables, i.e.\ Lemma \ref{lem:binom-tail} above, for $k>Tp$, we obtain
	\[
	\mathbb P(O_\mu\ge k\mid \eta^{(\alpha)})
	\le \left(\frac{eTp}{k}\right)^k.
	\]
	With $k=\theta=\kappa\log N_f$ this implies
	\[
	q_+:=\max_{\mu\notin S_\alpha}\mathbb P(x_\mu=1\mid \eta^{(\alpha)})
	\le \left(\frac{eTp}{\kappa\log N_f}\right)^{\kappa\log N_f}.
	\]
	By assumption we have $Tp\le C (\log N_f)^2/N_f$, hence
	\[
	\log q_+ \le -\kappa(\log N_f)^2 + O(\log N_f\log\log N_f),
	\]
	i.e.\ $q_+$ is super-polynomially small in $N_f$.
\end{proof}

We next show that, with high probability, every stored stroke contains at least
$\theta$ active entries. This guarantees that each true stroke belonging to the
target concept necessarily exceeds the activation threshold under the clean cue.

\begin{lemma}[{Lower bound for binomials}]
	\label{lem:no-light-strokes}
	Let $\{\xi^\mu\}_\mu$
		be i.i.d.\ {Bernoulli}$(p)$-random variables with $p=\log N_f/N_f$, and set
		$\theta=\kappa\log N_f$ with $\kappa\in(0,1)$. Then with
		$a=\tfrac12(1-\kappa)^2$,
		\[
		\mathbb P\Bigl(\min_{1\le\mu\le M} |\xi^\mu|_1 \ge \theta\Bigr)
		\ \ge\ 1 - M N_f^{-a}.
		\]
		In particular, if $M\le N_f^{a-\varepsilon}$ for some $\varepsilon>0$, then
		\[
		\min_{1\le\mu\le M} |\xi^\mu|_1 \ge \theta
		\]
		with probability $1-o(1)$ as $N_f\to\infty$.
	\end{lemma}
	
	\begin{proof}
		Fix $\mu$. Then $|\xi^\mu|_1=\sum_{i=1}^{N_f}\xi_i^\mu\sim\mathrm{Bin}(N_f,p)$ with
		$\mathbb E|\xi^\mu|_1=pN_f =\log N_f$.
		Let $\delta:=1-\kappa\in(0,1)$. By the standard Chernoff bound for binomials,
		\[
		\mathbb P\Bigl(|\xi^\mu|_1 \le (1-\delta)\mathbb E|\xi^\mu|_1\Bigr)
		\le \exp\!\left(-\frac{\delta^2}{2}\,\mathbb E|\xi^\mu|_1\right)
		= \exp\!\left(-\frac{(1-\kappa)^2}{2}\log N_f\right)
		= N_f^{-a},
		\]
		where $a=\tfrac12(1-\kappa)^2$.
		Since the strokes $\{|\xi^\mu|_1\}_{\mu=1}^M$ are identically distributed, a union bound gives
		\[
		\mathbb P\Bigl(\min_{1\le\mu\le M}|\xi^\mu|_1<\theta\Bigr)
		\le \sum_{\mu=1}^M \mathbb P\bigl(|\xi^\mu|_1<\theta\bigr)
		\le M N_f^{-a},
		\]
		which implies the first claim.
		If $M\le N_f^{a-\varepsilon}$, then $M N_f^{-a}\le N_f^{-\varepsilon}\to 0$ and the
		second claim follows.
	\end{proof}

\begin{proof}[Proof of Proposition \ref{prop:G-fixed-sparse}]
	Let
	\[
	E_C:=\{\|\eta^{(\alpha)}\|_1\le C\log N_f\},
	\]
	where $C>0$ is chosen as in Lemma~\ref{lem:cue-size-upper}. By that lemma,
	$\mathbb P(E_C^c)\xrightarrow[N_f\to\infty]{}0.$
	
	On the event $E_C$, Lemma~\ref{lem:false-positives} yields
	\[
	\max_{\mu\notin S_\alpha}\mathbb P(x_\mu=1\mid \eta^{(\alpha)})
	\le
	\overline q_+
	:=
	\exp\bigl(-c(\log N_f)^2\bigr)
	\]
	for some $c=c(\kappa,L,C)>0$ and all sufficiently large $N_f$.
	Since $M$ grows at most polynomially,
	\[
	(M-L)\overline q_+\xrightarrow[N_f\to\infty]{}0.
	\]	
	For the false negatives, note that for every $\mu\in S_\alpha$ one has
	$O_\mu(\eta^{(\alpha)})=|\xi^\mu|_1$,
	so that
	$q_-^{\mathrm{true}}
	=\mathbb P(|\xi^\mu|_1<\theta)$.
	By the same Chernoff estimate as in Lemma~\ref{lem:no-light-strokes},
	there exists $a=a(\kappa)>0$ such that
	\[
	q_-^{\mathrm{true}}\le N_f^{-a}
	\]
	for all sufficiently large $N_f$. Since $m=\lfloor \delta L\rfloor+1$ is fixed, it follows that
	\[
	\binom{L}{m}(q_-^{\mathrm{true}})^m\xrightarrow[N_f\to\infty]{}0.
	\]
	
	Applying Proposition~\ref{lem:G-fixed} on the event $E_C$ therefore gives
	\[
	\mathbb P\bigl(\mathcal G_\alpha(\delta,\rho)^c \mid E_C\bigr)\xrightarrow[N_f\to\infty]{}0.
	\]
	Hence
	\[
	\mathbb P\bigl(\mathcal G_\alpha(\delta,\rho)^c\bigr)
	\le
	\mathbb P\bigl(\mathcal G_\alpha(\delta,\rho)^c \mid E_C\bigr)+\mathbb P(E_C^c)
	\xrightarrow[N_f\to\infty]{}0,
	\]
	which proves the claim.
\end{proof}

\subsection{Proofs for exact recovery}

In this subsection we will deal with the proofs of the statements given in Subsection \ref{subsec:exact_recovery}.

\begin{proof}[Proof of Proposition~\ref{prop:strict-exact}]
	Work on the event from Lemma~\ref{lem:no-light-strokes}, which holds with probability
	$1-o(1)$ under the stated growth condition on $M$.
	
	Fix a concept $\alpha$ and consider the clean cue $\eta^{(\alpha)}$.
	If $\mu\in S_\alpha$, then 
	$O_\mu(\eta^{(\alpha)})=
	|\xi^\mu|_1\ge \theta$,
	so $x_\mu(\eta^{(\alpha)})=1$.
	If $\mu\notin S_\alpha$, Lemma~\ref{lem:false-positives} yields
	\[
	q_+ := \mathbb P(x_\mu=1\mid \eta^{(\alpha)})
	\le \exp\bigl(-c(\log N_f)^2\bigr).
	\]
	A union bound over $\mu\notin S_\alpha$ gives
	\[
	\mathbb P\bigl(\exists\,\mu\notin S_\alpha:\ x_\mu(\eta^{(\alpha)})=1\bigr)
	\le (M-L)q_+ \xrightarrow[N_f\to\infty]{}0
	\]
	for any polynomially growing $M$.
	Combining both parts yields exact recovery with probability $1-o(1)$.
\end{proof}

The following proof follows the same ideas: 

\begin{proof}[Proof of Proposition \ref{prop:strict-exact-used}]
	The proof is the same as that of Proposition~\ref{prop:strict-exact}, except that exact recovery is now required only on the set of used strokes. Thus the false-negative part is unchanged, since the target strokes are still precisely the elements of $S_\alpha$.
	
	For the false-positive part, however, we no longer need to exclude activations on all strokes in $\{1,\dots,M\}\setminus S_\alpha$, but only on the used strokes outside $S_\alpha$. Accordingly, the union bound is taken over $U\setminus S_\alpha$ rather than over $\{1,\dots,M\}\setminus S_\alpha$, where $U$ denotes the set of used strokes. Therefore the factor $M-L$ in the proof of Proposition~\ref{prop:strict-exact} is replaced by $|U|-L$.
	Hence the same argument yields
	\[
	\mathbb P\bigl(\text{restricted exact recovery fails}\bigr)
	\le
	L N_f^{-a}
	+
	\mathbb P\bigl(\|\eta^{(\alpha)}\|_1>C\log N_f\bigr)
	+
	(|U|-L)e^{-c(\log N_f)^2},
	\]
	for suitable constants $a>0$, $C>0$, and $c>0$. Under the stated assumptions, the right-hand side tends to zero, proving the claim.
\end{proof}

\section{Concepts with Variable Sizes}
\label{sec:var}

In Section~\ref{sec:fixed} we analysed hierarchical retrieval under the simplifying
assumption that every concept consists of exactly $L$ strokes. We now relax this
assumption and allow the concept sizes to be random. The main message remains the
same: successful concept retrieval does not require exact recovery of all strokes
at the intermediate layer. Rather, it is enough to control aggregate false
negatives and false positives, provided competing concepts are sufficiently well
separated.

The main new feature is that concept sizes now vary across $\alpha$. As a
result, the deterministic separation condition must take the size variability
into account. We therefore begin with a penalised score that compensates for
different concept sizes, establish a deterministic retrieval criterion on a size
window, and then combine it with overlap bounds and a control of the good event.
We conclude with a brief discussion of an alternative normalised score.

\subsection{Setup and notation}

Let $(L_\alpha)_{\alpha=1,\dots,P}$ be i.i.d.\ positive integer-valued random
variables. Conditionally on $L_\alpha=\ell$, the concept
\[
S_\alpha \subset \{1,\dots,M\}
\]
is chosen uniformly among all $\ell$-subsets, independently over $\alpha$.
In particular,
\[
|S_\alpha|=L_\alpha.
\]

For a fixed target concept $\alpha$, we again consider the clean cue
$\eta^{(\alpha)}$ defined in \eqref{eq:concept-or}, and the corresponding
stroke-layer output
\[
x=x(\eta^{(\alpha)}).
\]

As in Section~\ref{sec:fixed}, we define the numbers of false negatives and false
positives by
\[
F_-^{(\alpha)}
:=
\sum_{\mu\in S_\alpha}\mathbf 1_{\{x_\mu=0\}},
\qquad
F_+^{(\alpha)}
:=
\sum_{\mu\notin S_\alpha}\mathbf 1_{\{x_\mu=1\}}.
\]
For parameters $\delta\in(0,1)$ and $\rho\in\mathbb N_0$, we introduce the good
event
\[
\mathcal G_\alpha(\delta,\rho)
=
\Bigl\{
F_-^{(\alpha)}\le \delta L_\alpha
\Bigr\}
\cap
\Bigl\{
F_+^{(\alpha)}\le \rho
\Bigr\}.
\]
On this event, at least a fraction $1-\delta$ of the true strokes is active,
while at most $\rho$ false positives occur.

To account for varying concept sizes, we work primarily with the penalised score
\[
\mathrm{score}^{\mathrm{pen}}_\beta(x)
=
a\,|S_\beta\cap \supp(x)|
-
b\,L_\beta,
\qquad a,b>0,
\]
and decode according to
\[
\widehat\alpha(x)
=
\arg\max_{\beta\le P}\mathrm{score}^{\mathrm{pen}}_\beta(x).
\]
A normalised alternative score will be discussed in
Section~\ref{subsec:alternative-score}.

\subsection{Robust concept retrieval with a penalised score}
As in the fixed-size setting, the analysis begins with a deterministic
separation statement. The difference is that, because the concept sizes are now
variable, the separation threshold depends both on the target size and on the
size of the competitor.

The next lemma provides a deterministic sufficient condition for correct
decoding under the penalised score.

\begin{lemma}
	\label{lem:penalised-separation-randomL}
	
	Fix $\alpha$ and suppose that $\mathcal G_\alpha(\delta,\rho)$ holds.
	Then 
	\[
	\mathrm{score}^{\mathrm{pen}}_\alpha(x)
	\ge
	a(1-\delta)L_\alpha-bL_\alpha,
	\]
	and for every $\beta\neq\alpha$,
	\[
	\mathrm{score}^{\mathrm{pen}}_\beta(x)
	\le
	a\bigl(|S_\alpha\cap S_\beta|+\rho\bigr)-bL_\beta.
	\]
	
	Consequently,
	\[
	\mathrm{score}^{\mathrm{pen}}_\alpha(x)
	-
	\mathrm{score}^{\mathrm{pen}}_\beta(x)
	\ge
	a\bigl(
	(1-\delta)L_\alpha
	-
	|S_\alpha\cap S_\beta|
	-
	\rho	
	\bigr)
	-
	b(L_\alpha-L_\beta).
	\]
	
	In particular, if
	\[
	|S_\alpha\cap S_\beta|
	<
	(1-\delta)L_\alpha-\rho-\frac{b}{a}(L_\alpha-L_\beta),
	\]
	then
	\[
	\mathrm{score}^{\mathrm{pen}}_\alpha(x)
	>
	\mathrm{score}^{\mathrm{pen}}_\beta(x).
	\]
	
\end{lemma}

\begin{proof}
	On $\mathcal G_\alpha(\delta,\rho)$ at most $\delta L_\alpha$ strokes in
	$S_\alpha$ are inactive. Hence
	\[
	|S_\alpha\cap \supp(x)|
	=
	\sum_{\mu\in S_\alpha}x_\mu
	\ge
	(1-\delta)L_\alpha.
	\]
	Therefore
	\[
	\mathrm{score}^{\mathrm{pen}}_\alpha(x)
	=
	a|S_\alpha\cap \supp(x)|-bL_\alpha
	\ge
	a(1-\delta)L_\alpha-bL_\alpha.
	\]
	
	For $\beta\neq\alpha$, every active stroke counted in
	$S_\beta\cap \supp(x)$ either belongs to the true overlap
	$S_\alpha\cap S_\beta$ or is a ``false positive'' outside $S_\alpha$.
	Therefore
	\[
	|S_\beta\cap \supp(x)|
	\le
	|S_\alpha\cap S_\beta|+F_+^{(\alpha)}
	\le
	|S_\alpha\cap S_\beta|+\rho.
	\]
	Substituting this into the definition of
	$\mathrm{score}^{\mathrm{pen}}_\beta(x)$ proves the claim.
\end{proof}

Lemma~\ref{lem:penalised-separation-randomL} gives a deterministic comparison between the target concept $\alpha$ and a single competitor $\beta$. In order to obtain a usable retrieval criterion, we now make this bound uniform over all competitors whose sizes lie in a prescribed interval. This leads naturally to a size window $[\ell,u]$, on which both the target size and the competitor size can be controlled simultaneously.

More precisely, to obtain a uniform separation condition on a size window, we restrict attention to
 \[
 \mathcal A_{\ell,u}
 =
 \{\alpha\in\{1,\dots,P\}:\ \ell\le L_\alpha\le u\}.
 \]
 If $\alpha,\beta\in\mathcal A_{\ell,u}$, then
 \[
 L_\alpha\ge \ell
 \qquad\text{and}\qquad
 L_\alpha-L_\beta\le u-\ell.
 \]
 Hence Lemma~\ref{lem:penalised-separation-randomL} suggests the deterministic margin
 \begin{equation}
 	\label{eq:tlu}
 	t_{\ell,u}
 	:=
 	\Bigl\lfloor
 	(1-\delta)\ell-\rho-\frac{b}{a}(u-\ell)
 	\Bigr\rfloor.
 \end{equation}
 Indeed, if $\alpha,\beta\in\mathcal A_{\ell,u}$ and $t_{\ell,u}\ge0$, then on the event
 $\mathcal G_\alpha(\delta,\rho)$ one has
 \[
 \mathrm{score}^{\mathrm{pen}}_\alpha(x)>
 \mathrm{score}^{\mathrm{pen}}_\beta(x)
 \]
 whenever
 \[
 |S_\alpha\cap S_\beta|\le t_{\ell,u}.
 \]

Hence, on the event $\mathcal G_\alpha(\delta,\rho)$, correct decoding within the
window reduces to excluding competitors in $\mathcal A_{\ell,u}$ whose overlap
with the target exceeds $t_{\ell,u}$.

To turn the deterministic separation criterion into a probabilistic retrieval
statement, it remains to control accidental overlaps between competing concepts.
Conditionally on the concept sizes, this is again a simple combinatorial question.

\begin{lemma}
	\label{lem:overlap-random-sizes}
	
	For every $t\ge0$,
	\[
	\mathbb P\bigl(
	|S_\alpha\cap S_\beta|\ge t+1
	\mid
	L_\alpha=\ell_\alpha,\,
	L_\beta=\ell_\beta
	\bigr)
	\le
	\binom{\ell_\alpha}{t+1}
	\Bigl(\frac{\ell_\beta}{M-t}\Bigr)^{t+1}.
	\]
	
\end{lemma}

\begin{proof}
	Henceforth, we work conditionally on $(L_\alpha,L_\beta)=(\ell_\alpha,\ell_\beta)$. 
	If $|S_\alpha\cap S_\beta|\ge t+1$, then some subset
	$T\subset S_\alpha$ with $|T|=t+1$ is contained in $S_\beta$.
	There are $\binom{\ell_\alpha}{t+1}$ such subsets, so the union bound gives
	\[
	\mathbb P\bigl(
	|S_\alpha\cap S_\beta|\ge t+1
	\mid
	L_\alpha=\ell_\alpha,\,
	L_\beta=\ell_\beta
	\bigr)
	\le
	\binom{\ell_\alpha}{t+1}
	\mathbb P(T\subset S_\beta\mid L_\beta=\ell_\beta),
	\]
	where $T$ is any fixed $(t+1)$-subset.
	Now
	\[
	\mathbb P(T\subset S_\beta\mid L_\beta=\ell_\beta)
	=
	\frac{\binom{M-(t+1)}{\ell_\beta-(t+1)}}{\binom{M}{\ell_\beta}}
	=
	\prod_{k=0}^{t}\frac{\ell_\beta-k}{M-k}
	\le
	\Bigl(\frac{\ell_\beta}{M-t}\Bigr)^{t+1},
	\]
	which proves the claim.
\end{proof}

We now combine stroke-layer control, deterministic separation, and the
overlap estimate. Recall $t_{\ell,u}$ from \eqref{eq:tlu}.

The following main result of this section identifies the three basic mechanisms governing retrieval in the variable-size setting: stroke-layer accuracy, concentration of the target size inside the chosen window, and geometric separation from competing concepts through small overlaps.
\begin{theorem}[Concept retrieval for variable sizes]
	\label{thm:random-size-capacity}
	
	Fix $\delta\in(0,1)$, $\rho\in\mathbb N_0$, and $a,b>0$. Let $1\le \ell\le u$  and 
	assume $t_{\ell,u}\ge0$. For a size window $[\ell,u]$, define the restricted decoder
	\[
	\widehat\alpha_{\ell,u}(x)
	:=
	\arg\max_{\beta\in\mathcal A_{\ell,u}}
	\mathrm{score}^{\mathrm{pen}}_\beta(x).
	\]
	Then for every concept $\alpha$,
	\[
	\mathbb P\bigl(\widehat\alpha_{\ell,u}(x)\neq\alpha\bigr)
	\le
	\mathbb P\bigl(\mathcal G_\alpha(\delta,\rho)^c\bigr)
	+
	\mathbb P\bigl(L_\alpha\notin[\ell,u]\bigr)
	+
	\mathbb P\Bigl(
	\max_{\beta\in\mathcal A_{\ell,u}\setminus\{\alpha\}}
	|S_\alpha\cap S_\beta|
	\ge t_{\ell,u}+1
	\Bigr). 
	\]
    Consequently, by Lemma \ref{lem:overlap-random-sizes}, we have that
	\[\begin{split}
	\mathbb P\bigl(\widehat\alpha_{\ell,u}(x)\neq\alpha,\ \alpha\in\mathcal A_{\ell,u}\bigr)
	\le
	&~\mathbb P\bigl(\mathcal G_\alpha(\delta,\rho)^c,\ \alpha\in\mathcal A_{\ell,u}\bigr)
	+
	\mathbb P(L_\alpha\notin[\ell,u])
    \\
	&{}\quad+
	(P-1)\binom{u}{t_{\ell,u}+1}
	\Bigl(\frac{u}{M-t_{\ell,u}}\Bigr)^{t_{\ell,u}+1}.
	\end{split}
    \]
\end{theorem}
\begin{proof}
	The error event is contained in the union of the following three events:
	\begin{enumerate}
		\item $\mathcal G_\alpha(\delta,\rho)^c$,
		\item $L_\alpha\notin[\ell,u]$,
		\item 
		$
		\max_{\beta \in \mathcal A_{\ell ,u }\setminus \{\alpha\}}		|S_\alpha\cap S_\beta|\ge t_{\ell,u}+1.
		$
	\end{enumerate}
	
	Indeed, if none of these events occurs, then
	$\mathcal G_\alpha(\delta,\rho)$ holds, we have $L_\alpha\in[\ell,u]$, and
	for every $\beta\neq\alpha$ with $L_\beta\in[\ell,u]$,
	\[
	|S_\alpha\cap S_\beta|\le t_{\ell,u}.
	\]
	By Lemma~\ref{lem:penalised-separation-randomL}, this implies
	\[
	\mathrm{score}^{\mathrm{pen}}_\alpha(x)>
	\mathrm{score}^{\mathrm{pen}}_\beta(x)
	\qquad\text{for all }\beta\neq\alpha\text{ with }\beta\in\mathcal A_{\ell,u}.
	\]
	Therefore $\widehat\alpha_{\ell,u}(x)=\alpha.$
	
	The third event is bounded by Lemma~\ref{lem:overlap-random-sizes}
	together with the union bound over $\beta\neq\alpha$.
\end{proof}

 The next corollaries show how this abstract bound translates into simple asymptotic retrieval criteria for bounded and slowly growing size windows.

For an $M$-dependent size window $[\ell_M,u_M]$, let
\[
t_M:=t_{\ell_M,u_M}
=
\Bigl\lfloor
(1-\delta)\ell_M-\rho-\frac{b}{a}(u_M-\ell_M)
\Bigr\rfloor.
\]

\begin{corollary}
	\label{cor:general-window}
	
	Assume that
	\[
	\max_{\alpha \le P}\mathbb P\bigl(\mathcal G_\alpha(\delta,\rho)^c\bigr)\to0
	\]
	that
	\[
	\mathbb P(L_\alpha\notin[\ell_M,u_M])\to 0,
	\]
	and that $t_M$ is such that
	\[
	P\,\binom{u_M}{t_M+1}
	\Bigl(\frac{u_M}{M-t_M}\Bigr)^{t_M+1}\to 0.
	\]
	Then the \emph{restricted decoder}
	\[
	\widehat\alpha_{\ell_M,u_M}(x)
	:=
	\arg\max_{\beta\in\mathcal A_{\ell_M,u_M}}
	\mathrm{score}^{\mathrm{pen}}_\beta(x) 
	\]
	satisfies 
	\[
	\max_{\alpha\le P}
	\mathbb P\bigl(
	\widehat\alpha_{\ell_M,u_M}(x)\neq\alpha,\ 
	\alpha\in\mathcal A_{\ell_M,u_M}
	\bigr)\to 0.
	\]
\end{corollary}

The overlap contribution is controlled by
\[
P\,\binom{u_M}{t_M+1}
\Bigl(\frac{u_M}{M-t_M}\Bigr)^{t_M+1}.
\]
Thus the usefulness of Corollary~\ref{cor:general-window} depends on the
behaviour of the threshold $t_M$ associated with the chosen size window
$[\ell_M,u_M]$. In the special case where the threshold is eventually constant, say $t_M\equiv t$, the overlap bound is of order
\[
P\,u_M^{2(t+1)}M^{-(t+1)},
\]
up to constants. Hence in that regime, the bound is nontrivial whenever
$P$ grows slower than
\[
M^{t+1}/u_M^{2(t+1)}.
\]
More generally, when $t_M$ varies with $M$, one may use
Lemma~\ref{lem:simple-overlap-condition} below.

The simplest case is a bounded size window.

%
%
\begin{corollary}
	\label{cor:bounded-window}
Fix integers $1\le \ell\le u$, independent of $M$, and assume $t_{\ell,u}\ge0$.
Define
\[
\mathcal A_{\ell,u}
=
\{\alpha\le P:\ \ell\le L_\alpha\le u\}
\]
and
\[
\widehat\alpha_{\ell,u}(x)
=
\arg\max_{\beta\in\mathcal A_{\ell,u}}\mathrm{score}^{\mathrm{pen}}_\beta(x).
\]
If
	\[
	\max_{\alpha\le P}\mathbb P\bigl(\mathcal G_\alpha(\delta,\rho)^c\bigr)\to0,
	\qquad
	\mathbb P(L_\alpha\notin[\ell,u])\to0,
	\]
	and
	\[
	P=o(M^{t_{\ell,u}+1}),
	\]
	then again
	\[
	\max_{\alpha\le P}
	\mathbb P\bigl(
	\widehat\alpha_{\ell,u}(x)\neq\alpha,\ 
	\alpha\in\mathcal A_{\ell,u}
	\bigr)\to 0.
	\]
\end{corollary}

\begin{proof}
	Since $\ell$ and $u$ are fixed,
	\[
	\binom{u}{t_{\ell,u}+1}
	\Bigl(\frac{u}{M-t_{\ell,u}}\Bigr)^{t_{\ell,u}+1}
	=
	O\bigl(M^{-(t_{\ell,u}+1)}\bigr),
	\]
	so the overlap term in Theorem~\ref{thm:random-size-capacity} tends to
	zero under the stated assumption.
\end{proof}

The following lemma gives a convenient sufficient condition for the
overlap term to vanish when the threshold $t_M$ is bounded below by a fixed integer.
	\begin{lemma}
		\label{lem:simple-overlap-condition}
		Assume that there exists a fixed integer $t_*\ge0$ such that
		\[
		t_M\ge t_*
		\]
		for all sufficiently large $M$, and that $u_M=o(M)$. Then
		$
		P\,\frac{u_M^{2(t_*+1)}}{M^{t_*+1}}\to0
		$
		implies
		\[
		P\,\binom{u_M}{t_M+1}
		\Bigl(\frac{u_M}{M-t_M}\Bigr)^{t_M+1}\to0.
		\]
	\end{lemma}
	\begin{proof}
		Using the bound
		$
		\binom{u_M}{t_M+1}\le u_M^{t_M+1},
		$
		we obtain
		\[
		\binom{u_M}{t_M+1}
		\Bigl(\frac{u_M}{M-t_M}\Bigr)^{t_M+1}
		\le
		\Bigl(\frac{u_M^2}{M-t_M}\Bigr)^{t_M+1}.
		\]
		Since $u_M/M \to 0$, we have that $t_M<M/2$
		for all sufficiently large $M$. Because $t_M\le u_M$, it follows that
		$
		M-t_M\ge M-u_M\ge \frac M2.$
		Hence
		\[
		\frac{u_M^2}{M-t_M}\le 2\,\frac{u_M^2}{M}.
		\]  
		Therefore
		\[
		\binom{u_M}{t_M+1}
		\Bigl(\frac{u_M}{M-t_M}\Bigr)^{t_M+1}
		\le
		\Bigl(2\,\frac{u_M^2}{M}\Bigr)^{t_M+1}.
		\]
		Now, since
		$P\,\frac{u_M^{2(t_*+1)}}{M^{t_*+1}}\to 0$
		and $P\ge 1$, we also have
		$\frac{u_M^2}{M}\to 0$.
		Thus the base on the right-hand side is eventually less than $1$.
		Since $t_M\ge t_*$, we obtain
		\[
		\Bigl(2\,\frac{u_M^2}{M}\Bigr)^{t_M+1}
		\le
		\Bigl(2\,\frac{u_M^2}{M}\Bigr)^{t_*+1}
		=
		2^{t_*+1}\frac{u_M^{2(t_*+1)}}{M^{t_*+1}}.
		\]
		Multiplying by $P$ yields the claim.
	\end{proof}

As an illustration, we record a Poisson-type corollary.

\begin{corollary}
	\label{cor:poisson-window}
	
	Assume that $L_\alpha$ are i.i.d.\ with exponentially decaying tails,
	for instance Poisson or Poisson conditioned on $\{L_\alpha\ge1\}$.
	Since $L_\alpha$ has an exponentially decaying tail, there exist
	constants $c_1,c_2>0$ such that
	\[
	\mathbb P(L_\alpha>n)\le c_1 e^{-c_2 n}
	\]
	for all sufficiently large $n$. Hence, for $u_M=C\log M$,
	we have
	\[
	\mathbb P(L_\alpha>u_M)\le c_1 M^{-c_2 C}\to0.
	\]
	Suppose moreover that, for some fixed integer $t_*\ge0$,
	$
	t_M:=t_{\ell_*,u_M}\ge t_*
	$
	for all sufficiently large $M$, and that
	$
	\max_{\alpha\le P}\mathbb P\bigl(\mathcal G_\alpha(\delta,\rho)^c\bigr)\to0.
	$
	If
	\[
	P\,\frac{(\log M)^{2(t_*+1)}}{M^{t_*+1}}\to0,
	\]
	then
	\[
	\max_{\alpha\le P}\mathbb P(\widehat\alpha(x)\neq\alpha,\ 
	\alpha\in\mathcal A_{\ell,u})\to0.
	\]
	
\end{corollary}

For the penalised score, one must always additionally ensure that $t_M\ge0$.
This is automatic for bounded windows, but it may fail for growing
windows if the penalty ratio $b/a$ is kept fixed. 

\subsection{Control of the good event}
It remains to verify the stroke-layer input in Theorem~\ref{thm:random-size-capacity},
namely that the good event $\mathcal G_\alpha(\delta,\rho)$ holds with high probability.
The following lemma provides a convenient conditional bound on this event over a
size window.

\begin{lemma}[Control of the good event]
	\label{lem:G-window}
	Fix $\delta\in(0,1)$ and $\rho\in\mathbb N_0$. For $\ell\in\mathbb N$ let
	$
	m_\ell:=\lfloor \delta \ell\rfloor+1
	$
	and define the envelope probabilities
	\begin{equation}
		\label{eq:qpm-window}
		q_-(\ell)
		:=
		\max_{\mu\in S_\alpha}
		\mathbb P(x_\mu=0\mid \eta^{(\alpha)},L_\alpha=\ell),
		\quad\text{and}\quad
		q_+(\ell)
		:=
		\max_{\mu\notin S_\alpha}
		\mathbb P(x_\mu=1\mid \eta^{(\alpha)},L_\alpha=\ell).
	\end{equation}
	
	Then for every $\ell$,
	\[
	\mathbb P\bigl(\mathcal G_\alpha(\delta,\rho)^c \mid L_\alpha=\ell\bigr)
	\le
	\binom{\ell}{m_\ell}\,q_-(\ell)^{m_\ell}
	+
	\frac{\bigl((M-\ell)q_+(\ell)\bigr)^{\rho+1}}{(\rho+1)!}.
	\]
	Consequently, if for some window $[\ell_M,u_M]$,
	\begin{equation}\label{eq:sup1}
		\max_{\ell\in[\ell_M,u_M]}
		\binom{\ell}{m_\ell}\,q_-(\ell)^{m_\ell}\to0,
	\end{equation}
	and
	\begin{equation}\label{eq:sup2}
		\max_{\ell\in[\ell_M,u_M]}
		(M-\ell)q_+(\ell)\to0,
	\end{equation}
	then
	\[
	\max_{\ell\in[\ell_M,u_M]}
	\mathbb P\bigl(\mathcal G_\alpha(\delta,\rho)^c \mid L_\alpha=\ell\bigr)\to0.
	\]
	Moreover,
	\[
	\mathbb P\bigl(\mathcal G_\alpha(\delta,\rho)^c\bigr)
	\le
	\mathbb P(L_\alpha\notin[\ell_M,u_M])
	+
	\max_{\ell\in[\ell_M,u_M]}
	\mathbb P\bigl(\mathcal G_\alpha(\delta,\rho)^c \mid L_\alpha=\ell\bigr).
	\]
	Therefore, $\mathbb P\bigl(\mathcal G_\alpha(\delta,\rho)^c\bigr)\to 0$ whenever \eqref{eq:sup1}, \eqref{eq:sup2} and
	\[
	\mathbb P(L_\alpha\notin[\ell_M,u_M])\to 0
	\]
	all hold. 
\end{lemma}

\begin{proof}
	Since
	\[
	\mathcal G_\alpha(\delta,\rho)^c
	=
	\{F_-^{(\alpha)}>\delta L_\alpha\}
	\cup
	\{F_+^{(\alpha)}>\rho\},
	\]
	the union bound gives
	\[
	\mathbb P\bigl(\mathcal G_\alpha(\delta,\rho)^c \mid L_\alpha=\ell\bigr)
	\le
	\mathbb P\bigl(F_-^{(\alpha)}>\delta \ell \mid L_\alpha=\ell\bigr)
	+
	\mathbb P\bigl(F_+^{(\alpha)}>\rho \mid L_\alpha=\ell\bigr).
	\]
	The first term is bounded by
	\[
	\binom{\ell}{m_\ell}\,q_-(\ell)^{m_\ell},
	\]
	exactly as in the fixed-size case, by a union bound over all
	$m_\ell$-subsets of $S_\alpha$. The second term is bounded by
	\[
	\frac{\bigl((M-\ell)q_+(\ell)\bigr)^{\rho+1}}{(\rho+1)!},
	\]
	again as in the fixed-size case, by the factorial-moment estimate for
	false positives.
\end{proof}

\subsection{Alternative scoring: normalised overlap}
\label{subsec:alternative-score}

The penalised score is natural when the admissible concept sizes lie in a
bounded or moderately growing window. We conclude by briefly discussing a
normalised alternative whose deterministic separation condition is less
sensitive to the width of the size window.
More precisely, we consider the \emph{normalised overlap}:
\[
\mathrm{score}^{\mathrm{norm}}_\beta(x)
=
\frac{|S_\beta\cap \supp(x)|}{L_\beta}.
\]
Loosely speaking, in comparison with the penalised score, the corresponding
deterministic threshold for the normalised score depends only on the lower
cutoff $\ell$, making it more robust under substantial variability of concept
sizes.

\begin{lemma}
\label{lem:normalised-separation-randomL}
Fix $\alpha$ and suppose that $\mathcal G_\alpha(\delta,\rho)$ holds.
Then
\[
\mathrm{score}^{\mathrm{norm}}_\alpha(x)\ge 1-\delta,
\]
and for every $\beta\neq\alpha$,
\[
\mathrm{score}^{\mathrm{norm}}_\beta(x)
\le
\frac{|S_\alpha\cap S_\beta|+\rho}{L_\beta}.
\]
In particular, if
\[
|S_\alpha\cap S_\beta|+\rho<(1-\delta)L_\beta,
\]
then
\[
\mathrm{score}^{\mathrm{norm}}_\alpha(x)>
\mathrm{score}^{\mathrm{norm}}_\beta(x).
\]
\end{lemma}

\begin{proof}
On $\mathcal G_\alpha(\delta,\rho)$ at most $\delta L_\alpha$ true strokes
are missing, so $
|S_\alpha\cap \supp(x)|\ge (1-\delta)L_\alpha.$
Hence
\[
\mathrm{score}^{\mathrm{norm}}_\alpha(x)
=
\frac{|S_\alpha\cap \supp(x)|}{L_\alpha}
\ge 1-\delta.
\]
Moreover, every active stroke counted in $S_\beta\cap \supp(x)$ either
belongs to $S_\alpha\cap S_\beta$ or is a false positive outside
$S_\alpha$, so
\[
|S_\beta\cap \supp(x)|\le |S_\alpha\cap S_\beta|+\rho.
\]
Dividing by $L_\beta$ gives the stated bound.
\end{proof}

\begin{remark}
\label{rem:pen-vs-norm}
For the penalised score, the deterministic threshold on a size window
$[\ell,u]$ is
\[
t_{\ell,u}
=
\Bigl\lfloor
(1-\delta)\ell-\rho-\frac{b}{a}(u-\ell)
\Bigr\rfloor.
\]
Hence a wide size window makes separation more difficult. In particular,
if $\ell$ is fixed and $u\to\infty$ while $b/a$ is kept fixed, then
$t_{\ell,u}\to-\infty$, so the deterministic separation condition becomes
vacuous.

By contrast, for the normalised score one has on
$\mathcal G_\alpha(\delta,\rho)$
\[
\mathrm{score}^{\mathrm{norm}}_\alpha(x)\ge 1-\delta,
\qquad
\mathrm{score}^{\mathrm{norm}}_\beta(x)
\le
\frac{|S_\alpha\cap S_\beta|+\rho}{L_\beta}.
\]
Thus, on the window $[\ell,u]$, correct separation is guaranteed whenever
\[
|S_\alpha\cap S_\beta|
<
(1-\delta)\ell-\rho.
\]
This threshold depends only on the lower cutoff $\ell$, not on the width
$u-\ell$ of the window. Therefore the normalised score is more robust
than the penalised score under substantial variability of concept sizes.

The price one pays is that very small competing concepts may be
problematic, since normalisation by $L_\beta$ makes them comparatively
easy to match. For this reason, a lower size cutoff remains essential for
the normalised score as well.
\end{remark}

\subsection*{Acknowledgements} 
MH thanks the Aarhus Institute for Advanced Studies for their kind hospitality during a visit in July 2025.

ML's research was funded by the Deutsche Forschungsgemeinschaft (DFG, German Research Foundation) under Germany's Excellence Strategy EXC 2044/2 - 390685587, Mathematics M\"unster: \emph{Dynamics-Geometry-Structure}.

\bibliographystyle{abbrv}
\bibliography{LiteraturDatenbank}  
	
\end{document}